\documentclass[11pt,twoside]{article}
\usepackage{amsmath,amssymb,epsfig}
\usepackage{amsmath,amssymb,epsfig,color,graphicx,subfigure}
\usepackage{multirow}

\newtheorem{proposition}{Proposition}[section]
\newtheorem{theorem}[proposition]{Theorem}
\newtheorem{lemma}[proposition]{Lemma}
\newtheorem{corollary}[proposition]{Corollary}

\newtheorem{example}[proposition]{Example}
\newtheorem{algorithm}[proposition]{Algorithm}
\newtheorem{assumption}[proposition]{Assumption}

\newenvironment{proof}{{\noindent \bf Proof:}}{\hfill $\fbox{}$ \vspace*{5mm}}
\numberwithin{equation}{section}

\newcommand{\bx}{{\bf x}}

\newcommand{\ch}{\mathcal{H}}

\newcommand{\cm}{\mathcal{M}}

\newcommand{\ct}{\mathcal{T}}

\newcommand{\dist}{{\rm dist}}

\newcommand{\grad}{{\rm grad\,}}

\newcommand{\qf}{{\rm qf}}
\newcommand{\sym}{{\rm sym}}
\newcommand{\skw}{{\rm skew}}

\newcommand{\st}{{\rm St}}
\newcommand{\tr}{{\rm tr}}

\newcommand{\R}{{\mathbb R}}

\newcommand{\Rn}{{\mathbb R}^n}

\newcommand{\BE}{\begin{equation}}
\newcommand{\EE}{\end{equation}}

\newcommand{\normmm}[1]{{\vert\kern-0.25ex \vert\kern-0.25ex \vert #1
    \vert\kern-0.25ex \vert\kern-0.25ex\vert}}

\begin{document}

\title{ A Riemannian Derivative-Free Polak-Ribi\'ere-Polyak Method for Tangent Vector Field}
\author{ Teng-Teng Yao\thanks{Department of Mathematics, School of Sciences, Zhejiang University of Science and Technology, Hangzhou 310023, People's Republic of China (yaotengteng718@163.com). The research of this author is supported by the National Natural Science Foundation of China (No. 11701514).}
\and Zhi Zhao\thanks{Department of Mathematics, School of Sciences, Hangzhou Dianzi University, Hangzhou 310018, People's Republic of China (zhaozhi231@163.com). The research of this author is supported by the National Natural Science Foundation of China (No. 11601112).}
\and Zheng-Jian Bai\thanks{Corresponding author. School of Mathematical Sciences and Fujian Provincial Key Laboratory on Mathematical Modeling \& High Performance Scientific Computing,  Xiamen University, Xiamen 361005, People's Republic of China (zjbai@xmu.edu.cn). The research of this author is partially supported by the National Natural Science Foundation of China (No. 11671337), the Natural Science Foundation of Fujian Province of China (No. 2016J01035), and the Fundamental Research Funds for the Central Universities (No. 20720180008).}
\and Xiao-Qing Jin\thanks{Department of Mathematics, University of Macau, Macao, People's Republic of China (xqjin@umac.mo). The research of this author is supported by the research grant MYRG2016-00077-FST from University of Macau.}
 }
\maketitle

\begin{abstract}
This paper is concerned with the problem of finding a zero of a tangent vector field on a Riemannian manifold.
We first reformulate the problem as an equivalent Riemannian optimization problem. Then we propose a Riemannian derivative-free Polak-Ribi\'ere-Polyak method for solving the Riemannian optimization problem, where a non-monotone line search is employed. The global convergence of the proposed method is established under some mild assumptions. To further improve the efficiency, we also provide a hybrid method, which combines the proposed  geometric method with the Riemannian Newton method.  Finally, some numerical experiments are reported to illustrate the efficiency of the proposed method.

\end{abstract}

{\bf Keywords.}  Tangent vector field, Riemannian manifold, Polak-Ribi\'ere-Polyak method, non-monotone line search.

\vspace{3mm}
{\bf 2010 AMS subject classifications.} 65K05, 90C30, 90C56.

\section{Introduction}\label{sec1}
Let $\cm$ be a finite-dimensional Riemannian manifold and let $\langle \cdot, \cdot \rangle$  be the Riemannian metric on $\cm$ with its induced norm $\|\cdot \|$. Let $\nabla$ denote the Riemannian connection on $\cm$ induced by the Riemannian metric $\langle \cdot, \cdot \rangle$. Let $T_X\cm$ be the tangent space of $\cm$ at a point $X\in\cm$ and  $T\cm:=\cup_{X\in\cm}T_X\cm$ be the tangent bundle of $\cm$. In this paper, we aim to find a zero of  a continuously differentiable tangent vector field $F:\mathcal{M} \rightarrow T\cm$, i.e.,  find $X \in \mathcal{M}$ such that
\BE\label{prb}
F(X) = 0_X,
\EE
where $0_X$ is the zero tangent vector of $T_X\cm$.

Such smooth tangent vector fields arise in many applications such as geodesic convex optimizations on Riemannian manifolds where the gradients of the convex objective functions are geodesic monotone vector fields \cite{DFL00,NE99},  statistical principal component analysis where the Oja's flow leads to the Oja's vector field \cite{OJA82,OJA89}, the discretized Kohn-Sham (KS) total energy minimization in electronic structure calculations \cite{CD14, M04, SCS10}, and the trace ratio optimization in the linear discriminant analysis (LDA) for dimension reduction \cite{NB10, ZL14, ZL15} where the corresponding eigenvector-dependent nonlinear eigenvalue problems are smooth tangent vector fields, etc.

In particular, for multivalued monotone tangent vector fields on Hadamard manifolds, several proximal point algorithms have been proposed in \cite{FO02,LLM09,TH13,WLLY15,WLLY16}, where the  convergence analysis is investigated under some different assumptions.
However, these proximal point algorithms  are mainly restricted to finding zeros of monotone tangent vector fields.

For smooth tangent vector fields on general Riemannian manifolds, Riemannian Newton  method was widely studied (see for instance \cite{AMS08,ADM02,DPM03,LW05}). In  \cite[Section 6.1]{AMS08}, Absil et al. presented a geometric Newton method for solving  (\ref{prb}): Given current $X_k\in\cm$, solve the Riemannian Newton equation
\[
JF(X_k)[ \Delta X_k ] = -F(X_k)
\]
for $\Delta X_k\in T_{X_k}\cm$ and  set
\[
X_{k+1} = R_{X_k}(\Delta X_k),
\]
where $R$ is a retraction defined on $\cm$ \cite[Definition 4.1.1]{AMS08} and for $X\in\cm$, $R_{X}$ is the restriction of $R$ to $T_{X}\cm$. Here, $JF(X)$ denotes the Jacobian of $F$ at a point $X\in \cm$, which is a linear operator from $T_X\cm$ to  $T_X \cm$ defined by \cite[p.111]{AMS08}
\[
JF(X)[\xi_X]:=\nabla_{\xi_X}F,\quad\forall \xi_X\in T_X\cm.
\]
With respect to the Riemannian metric $\langle \cdot, \cdot \rangle$, the adjoint  $(JF(X))^*: T_X\cm \to  T_X \cm$ of $JF(X)$  is defined by
\[
\langle \xi_X, (JF(X))^*[\eta_X] \rangle=\langle JF(X)[\xi_X], \eta_X \rangle,\quad \forall \xi_X,\eta_X\in T_X\cm.
\]

The quadratic convergence of the Riemannian Newton method was established under the nonsingularity assumption of the Jacobian of $F$ at a solution point \cite[Theorem 6.3.2]{AMS08}.
In \cite{AILH09}, Absil et al. also proposed  a geometric  Newton method for finding a zero of Oja's vector field.

The advatange of a geometric  Newton method lies in its quadratic convergence. However, it is often computationally costly to solve the Riemannian Newton equation, especially when the Jacobian is ill-conditioned. In the case of  large-scale problems, the Jacobians of some tangent vector fields
(e.g., monotone tangent vector fields on Hadamard manifolds) may not be easily available. Finally, the convergence of the Riemannian Newton method also depends on the starting point.
Therefore, it is indispensable to find an efficient globally convergent Jacobian-free method for solving (\ref{prb}), especially for large-scale problems.

In recent years, some derivative-free optimization methods have been proposed for solving nonlinear systems of equations of the form of $G(\bx)={\bf 0}$ defined on Euclidean spaces \cite{CXH09,CMR06,CR03,FN17,ML14,Yu10,Yu11}, where $G:\Rn\to\Rn$ is a continuously differentiable mapping.  These methods use $\pm G(\bx_k)$ at the current iterate $\bx_k$ as a search direction and their global convergence are guaranteed by using some non-monotone  line search techniques.
These methods need not to form the Jacabian matrices and require a small storage space and thus are applicable to solving large-scale nonlinear systems of equations. Sparked by this, in this paper, we propose a Riemannian Derivative-Free Polak-Ribi\'ere-Polyak (PRP) conjugate gradient method for solving (\ref{prb}). The global convergence is established under some assumptions. We apply  the proposed method to finding zeros of Oja's vector fields, the tangent vector field corresponding to the trace ratio optimization problem, and monotone tangent vector fields on Hadamard manifolds accordingly. Finally, we combine the  proposed method with the Riemannian Newton method to get a solution of high accuracy.

The remaining part of this paper is organized as follows. In section \ref{sec2}, we present a  Riemannian derivative-free PRP conjugate gradient method for solving (\ref{prb}). In section \ref{sec3},
we give  the global convergence of the proposed method under some basic assumptions. In section \ref{sec4},  the proposed method is used to find zeros of tangent vector fields for some practical applications. In section \ref{sec5}, we present a hybrid method. Finally some concluding remarks are given in section \ref{sec6}.

\section{A Riemannian Derivative-free Polak-Ribi\'ere-Polyak Method}\label{sec2}
We first recall the Riemannian nonlinear conjugate gradient method for solving the following optimization problem
\BE\label{g}
\begin{array}{cc}
\min & \displaystyle g(Z)\\[2mm]
\mbox{subject to (s.t.)} & Z\in \cm,
\end{array}
\EE
where $g:\cm\to \R$ is a continuously differentiable function. A nonlinear conjugate gradient method aims to update the current iterate $Z_k\in\cm$ by
\BE\label{NONCG}
Z_{k+1}=R_{Z_k}(\alpha_k\Delta Z_k),
\EE
where the step length $\alpha_k$ is determined by a line search. The search direction $\Delta Z_k\in T_{Z_k}\cm$ is given by
\BE\label{NONCOMBINATOIN}
\Delta Z_k=\left\{
\begin{array}{ll}
-\grad g(Z_k), &\mbox{if $k=0$},\\
-\grad g(Z_k)+\beta_k\ct_{\alpha_{k-1}\Delta Z_{k-1}} \Delta Z_{k-1}, &\mbox{if $k\ge 1$},
\end{array}
\right.
\EE
where $\beta_k$ is a scalar, $\grad g(Z_k)$ is the  Riemannian gradient of $g$ at the point $Z_k$, and $\ct$ is a vector transport associated with the retraction $R$ \cite[Definition 8.1.1]{AMS08}.  In particular, for the Riemannian PRP method in  \cite[p.182]{AMS08},  the parameter $\beta_k$ is given by
\BE\label{NONCONEFFICIENTS}
\beta_{k}=\frac{\langle \grad g(Z_k), \grad g(Z_{k}) - \ct_{\alpha_{k-1}\Delta Z_{k-1}} \grad g(Z_{k-1}) \rangle} {\|{\rm grad}\; g(Z_{k-1})\|^{2}}.
\EE
For more  Riemannian nonlinear conjugate gradient methods, one may refer to \cite{AMS08,EAS98,RW12,SATOIWAI15,YBZC16,ZJB16,Z17}.
However, for any Riemannian nonlinear conjugate gradient method for solving problem (\ref{g}),
the Riemannian gradient of $g$ is needed.

To solve (\ref{prb}), it is natural to consider the  following minimization problem
\BE\label{fF}
\begin{array}{cc}
\min & \displaystyle f(X):=\frac{1}{2}\|F(X)\|^{2}\\[2mm]
\mbox{s.t.} & X\in \cm.
\end{array}
\EE
Since $F:\cm \to T\cm$ is continuously differentiable, the function $f:\cm \to \R$ is also continuously differentiable. By the definition of Riemannian gradient and using the compatibility of Riemannian connection $\nabla$ with the Riemannian metric $\langle \cdot, \cdot \rangle$, we have
\begin{eqnarray*}
{\rm D}f(X)[\xi_X] &=& \frac{1}{2}\big(\langle \nabla_{\xi_X}F, F(X) \rangle + \langle F(X),
\nabla_{\xi_X}F \rangle\big)\\
&=& \langle \nabla_{\xi_X}F, F(X) \rangle = \langle JF(X)[\xi_X], F(X) \rangle\\
&=& \langle \xi_X, (JF(X))^*[F(X)] \rangle = \langle \xi_X, {\rm grad}f(X)\rangle,
\end{eqnarray*}
for all $\xi_X\in T_X\cm$. Thus,
\BE\label{GRAVEC}
{\rm grad}f(X) = (JF(X))^*[F(X)].
\EE

In order to apply the Riemannian PPR method determined by (\ref{NONCG}), (\ref{NONCOMBINATOIN}), and (\ref{NONCONEFFICIENTS}) for solving problem (\ref{fF}), we need the  Riemannian gradient of $f$.
By using (\ref{GRAVEC}), to calculate the Riemannian gradient of $f$ at the current iterate $X_k$, we need to compute the adjoint of the Jacobian of $F$ at $X_k$. If the Jacobian of $F$ is not available or numerically expensive to calculate, then it is unsuitable  to  directly apply a Riemannian nonlinear conjugate gradient method to problem (\ref{fF}).

In the following, we propose a derivative-free PRP method for solving  (\ref{prb}). This is motivated by the derivative-free PRP method for solving a large-scale nonlinear system of equations of the form $G(\bx)={\bf 0}$ with  $G:\Rn\to\Rn$ being continuously differentiable  \cite{ML14}, where the search direction uses $G(\bx_k)$ and $G(\bx_{k-1})$ at the current iterate $\bx_k$ and the previous iterate $\bx_{k-1}$ and a non-monotone line search is used. In particular, we use the PRP method defined by (\ref{NONCG}), (\ref{NONCOMBINATOIN}), and (\ref{NONCONEFFICIENTS}) to  problem (\ref{fF}), where the Riemannian gradients of $f$ at the current iterate $X_k$ is replaced by the values of the tangent vector field $F$ at $X_k$, and a Riemannian nonmonotnoe line search is employed.
We now describe a Riemannian derivative-free PRP algorithm for solving (\ref{prb})  as follows.

\begin{algorithm}\label{RDF-PRP}
{\rm (A Riemannian derivative-free PRP method (RDF-PRP))}
\begin{description}
\item [{\rm Step 0.}]
Choose an initial point $ X_{0} \in \mathcal{M}$, $\bar{\epsilon}>0$, $t_1,t_2>0$, $0<\rho<1$, $0<\lambda_{\min}<\lambda_{\max}<1$, $0<\alpha_{\min}\le\alpha\le\alpha_{\max}$. Let $k:=0$, $\Gamma_{0}:=f(X_{0})$, $\Phi_{0}:=1$.
Select a positive sequence $\{\delta_{k}\}$ such that
\BE\label{n}
\sum_{k=0}^{\infty}\delta_{k}=\delta<\infty.
\EE

\item [{\rm Step 1.}] If $\|F(X_{k})\|\le\bar{\epsilon}$, then stop. Otherwise, go to {\rm Step 2}.

\item [{\rm Step 2.}] Set
\BE\label{CONJUG}
\Delta X_{k}:=\left\{
\begin{array}{ll}
- F(X_{0}) & \mbox{if $k=0$},\\[2mm]
- F(X_{k}) + \beta_{k}\ct_{\Delta Z_{k-1}}\Delta X_{k-1}, & \mbox{if $k\ge 1$},
\end{array}
\right.
\EE
where
\BE\label{def:betak}
\beta_{k}  :=  \frac{\langle F(X_{k}), Y_{k-1} \rangle}
{ \|F(X_{k-1}) \|^2 }, \quad
 Y_{k-1} := F(X_{k})-\ct_{\Delta Z_{k-1}}F( X_{k-1}).
\EE
\item [{\rm Step 3.}]  Determine $\alpha_k = \max \{ \alpha\rho^j, j = 0,1,2,\ldots \}$ such that \\
if
\BE\label{BACKTR1}
f( R_{X_k}(\alpha_k \Delta X_k) )\leq \Gamma_{k}+\delta_{k}-t_1\alpha_k^2\|\Delta X_k\|^{2}-t_2\alpha_k^2f(X_k),
\EE
then set
\BE\label{RER1}
\Delta Z_k := \alpha_k \Delta X_k, \quad X_{k+1} := R_{X_{k}} (\Delta Z_k).
\EE
Else if
\BE\label{BACKTR2}
f( R_{X_k}(-\alpha_k \Delta X_k) )\leq \Gamma_{k}+\delta_{k}-t_1\alpha_k^2\|\Delta X_k\|^{2}-t_2\alpha_k^2f(X_k),
\EE
then set
\BE\label{RER2}
\Delta Z_k := -\alpha_k \Delta X_k, \quad X_{k+1} := R_{X_{k}} (\Delta Z_k).
\EE

\item [{\rm Step 4.}] Choose $\lambda_{k}\in[\lambda_{\min},\lambda_{\max}]$ and compute
\BE\label{QC}
\Phi_{k+1}=\lambda_{k}\Phi_{k}+1, \quad \Gamma_{k+1}=\frac{\lambda_{k}\Phi_{k}(\Gamma_{k}+\eta_{k})+f(X_{k+1})}{\Phi_{k+1}}.
\EE

\item [{\rm Step 5.}] Replace $k$ by $k+1$ and go to {\rm Step 1}.
\end{description}
\end{algorithm}

We point out that the non-monotone line search in Step 3 of Algorithm {\rm \ref{RDF-PRP}} can be seen as a generalization of that in \cite{CL09,ML14}. Let
\[
\Lambda_{k}:=\frac{\sum_{j=1}^{k}(f(X_{k})+j\delta_{j-1})}{k+1} \quad\mbox{and}\quad \delta_{-1} = 0.
\]
By following the similar proof of \cite[Lemma 2.2]{CL09}, for any choice of $\lambda_{k}\in[0,1]$, we have for all $k\geq 0$ that
\BE\label{PINEQUALITY}
f(X_{k}) \leq \Gamma_{k} \leq \Lambda_{k}, \quad \Gamma_{k+1} \leq \Gamma_{k} + \delta_k.
\EE
Then  condition (\ref{BACKTR1}) or (\ref{BACKTR2}) holds for some $\alpha_{k}$. This shows that the line search step in Algorithm \ref{RDF-PRP} is well-defined.

\section{Convergence Analysis}\label{sec3}

In this section, we establish the global convergence of Algorithm \ref{RDF-PRP}.
To facilitate the analysis, we  define the pullback $\widehat{f}: T\cm\to \R$ of $f:\cm \to \R$ through $R$
 by \cite[p.55]{AMS08}
\[
\hat{f} (\xi):= f(R(\xi)),\quad \forall \xi\in T\cm.
\]
For $X\in\cm$, let $\widehat{f}_X$ denote the restriction of $\widehat{f}$ to $T_X \cm$, i.e.,
\[
\hat{f}_X(\xi_X):= f(R_X(\xi_X)),\quad \forall \xi_X\in T_X\cm.
\]
We also need the following assumptions.

\begin{assumption}\label{ASSUM}
\begin{enumerate}
\item The level set $\Omega:=\{X\in \cm \; | \; f(X)\leq f(X_{0})+\delta\}$ is bounded, where $\delta$ is a constant
defined by {\rm (\ref{n})}.
\item In some neighborhood $V$ of $\Omega$, $F$ is continuously differentiable and is Lipschitz continuous
with respect to the vector transport $\ct$,
i.e., there is a constant $L > 0$ such that
\BE\label{bound}
\| F(R_X(\xi_X)) - \ct_{\xi_X}F(X) \| \leq L \cdot{\rm dist}(X,R_X(\xi_X)),
\EE
for all $X\in V$ and $\xi_X\in T_X\cm$ with $R_X(\xi_X)\in V$.
\item The vector transport $\ct$ is bounded, i.e., there exists a constant $C>0$ such that
\BE\label{TRANSPORT}
\| \ct_{\eta_X}\xi_X \| \leq C\cdot \|\xi_X\|,
\EE
for all $X\in \cm$ and $\xi_X,\eta_X\in T_X\cm$.
\end{enumerate}
\end{assumption}
Under Assumption {\rm \ref{ASSUM}}, the tangent vector field $F$ is bounded on $\Omega$, i.e.,
there exists a constant $\tau_{1}>0$ such that
\BE\label{Frx}
\|F(X)\|\leq\tau_{1}, \quad \forall X\in\Omega.
\EE
By using the continuity of $f$ and  Assumption {\rm \ref{ASSUM}}, it is easy to see that
the level set $\Omega$ is closed  and bounded and  thus $\Omega$ is a compact subset of $\cm$.
According to Corollary 7.4.6 in \cite{AMS08}, there exist
two scalars $\nu >0$ and $\mu >0$ such that
\BE\label{retr:bd}
\nu\| \xi_X \| \geq  {\rm dist}\big(X,R_X (\xi_X) \big),
\EE
for $X\in \Omega$ and $\xi_X\in T_X\cm$ with $\|\xi_X\| \leq \mu$.
If the vector tansport $\ct$ is chosen as the parallel translation, then the inequality in (\ref{TRANSPORT}) holds as an equality with $C=1$.
Specially, if $\cm$ is an embedded Riemannian submanifold of a Euclidean space and $\ct$ is defined through orthogonal projection
by the formula (8.10) in \cite[p.174]{AMS08}, then the inequality in (\ref{TRANSPORT}) holds with $C=1$.

To establish the global convergence Algorithm {\rm \ref{RDF-PRP}}, we need the following preliminary lemma whose proof is similar to that of Lemma 3.2 and Lemma 3.3 in \cite{ML14}, and thus we omit it here.
\begin{lemma}\label{lem:2}
Suppose  Assumption {\rm \ref{ASSUM}} is satisfied. Then the sequence $\{X_{k}\}$ generated by Algorithm {\rm \ref{RDF-PRP}}
is contained in $\Omega$. In addition, we have
\[
\lim\limits_{k \to \infty} \alpha_{k}\|\Delta X_{k}\|=0 \quad and \quad \lim\limits_{k \to \infty} \alpha_{k}^{2}f( X_{k})=0.
\]
\end{lemma}

For the search directions $\{\Delta X_{k}\}$ generated by Algorithm {\rm \ref{RDF-PRP}}, we have the following result. The proof can be can seen as a generalization of  \cite[Lemma 3.4]{ML14}.
\begin{lemma}\label{lem:4}
Suppose Assumption {\rm \ref{ASSUM}} is satisfied and Algorithm {\rm \ref{RDF-PRP}} generates infinite sequences $\{X_{k}\}$ and $\{\Delta X_{k}\}$.  If the sequence $\{\|F(X_k)\|\}$ is bounded below by a constant $\tau>0$, i.e.,
\BE\label{fr}
\|F(X_{k})\|\geq\tau,\quad \forall k\geq0,
\EE
then there exists a constant $T > 0$ such that
\BE\label{xb}
\|\Delta X_{k}\|\leq T, \quad \forall k\geq0,
\EE
and
\BE\label{bex}
\lim\limits_{k\rightarrow\infty}\beta_{k}\|\Delta X_{k-1}\|=0, \quad \forall k\geq0.
\EE
\end{lemma}
\begin{proof}
We first prove  (\ref{xb}). From (\ref{def:betak}), (\ref{RER1}), (\ref{RER2}), (\ref{bound}), and (\ref{retr:bd}) it follows that for all $k$ sufficiently large,
\begin{eqnarray}\label{y1}
\|Y_{k-1}\|&=&\|F(X_{k})-\ct_{\Delta Z_{k-1}}F(X_{k-1})\|  \nonumber\\
&=&\|F(R_{X_{k-1}}(\Delta Z_{k-1}))-\ct_{\Delta Z_{k-1}}F(X_{k-1})\| \nonumber\\
&\leq&L\cdot \dist(X_{k},X_{k-1}) =L \cdot \dist(R_{X_{k-1}}(\Delta Z_{k-1}),X_{k-1})\nonumber\\
&\leq&L\nu\|\Delta Z_{k-1}\| = L\nu\alpha_{k-1}\|\Delta X_{k-1}\|.
\end{eqnarray}
It follows from (\ref{CONJUG}), (\ref{def:betak}), (\ref{TRANSPORT}), (\ref{Frx}), and (\ref{y1}) that for all $k$ sufficiently large,
\begin{eqnarray}\label{Pflm:1}
\|\Delta X_{k}\|&=& \|-F(X_{k})+\frac{\langle F(X_{k}),Y_{k-1}\rangle}{\|F(X_{k-1})\|^{2}}\ct_{\Delta Z_{k-1}}\Delta X_{k-1}\| \nonumber\\
&\leq& \|F(X_{k})\|+\frac{\|F(X_{k})\|\cdot\|Y_{k-1}\|}{\|F(X_{k-1})\|^{2}}\|\ct_{\Delta Z_{k-1}}\Delta X_{k-1}\|  \nonumber\\
&\leq&\tau_{1}+\frac{\tau_{1}CL\nu\alpha_{k-1}\|\Delta X_{k-1}\|}{\tau^{2}}\|\Delta X_{k-1}\|.
\end{eqnarray}
By Lemma \ref{lem:2}, for any constant $\pi\in(0,1)$, there exists an index $k_{0}>0$ such that
\[
\frac{\tau_{1}CL\nu\alpha_{k-1}\|\Delta X_{k-1}\|}{\tau^{2}}<\pi,\quad \forall k>k_{0}.
\]
This, together with  (\ref{Pflm:1}), yields for all $k>k_{0}$,
\begin{eqnarray*}
\|\Delta X_{k}\|&\leq&\tau_{1}+\pi\|\Delta X_{k-1}\|\\
&\leq&\tau_{1}(1+\pi+\pi^{2}+\cdots+\pi^{k-k_{0}+1})+\pi^{k-k_{0}}\|\Delta X_{k_{0}}\|\\
&\leq& \frac{\tau_{1}}{1-\pi}+\|\Delta X_{k_{0}}\|.
\end{eqnarray*}
Hence, (\ref{xb}) holds by setting $T :=\max \{\|\Delta X_{1}\|,\|\Delta X_{2}\|,\ldots,\|\Delta X_{k_{0}}\|,\frac{\tau_{1}}{1-\pi}+\|\Delta X_{k_{0}}\| \}$.

Next, we prove  (\ref{bex}). By using Lemma \ref{lem:2}, (\ref{bound}), (\ref{fr}), (\ref{xb}), and (\ref{y1})  we have for all $k$ sufficiently large that
\begin{eqnarray*}
|\beta_{k}|\|\Delta X_{k-1}\|&=& \frac{|\langle F(X_{k}),Y_{k-1}\rangle|}{\|F(X_{k-1})\|^{2}}\|\Delta X_{k-1}\|\\
&\leq& \frac{\| F(X_{k})\|\cdot \|Y_{k-1}\|}{\|F(X_{k-1})\|^{2}}\|\Delta X_{k-1}\|
\leq \frac{\|F(X_{k})\|L\nu\alpha_{k-1}\|\Delta X_{k-1}\|}{\|F(X_{k-1})\|^{2}}\|\Delta X_{k-1}\| \\
&\leq& L\nu\frac{\|F(X_{k})\|}{\|F(X_{k-1})\|^{2}}\alpha_{k-1}\|\Delta X_{k-1}\|^{2}
\leq LT\nu\frac{\tau_{1}}{\tau^{2}}\cdot\alpha_{k-1}\|\Delta X_{k-1}\|.
\end{eqnarray*}
This, together with Lemma \ref{lem:2}, yields (\ref{bex}).
\end{proof}

On the global convergence of Algorithm \ref{RDF-PRP},  we have  the following theorem.
The proof is a generalization of Theorem 3.5 in \cite{ML14} and Theorem 1 in \cite{CMR06}.
\begin{theorem}\label{Thm:1}
Suppose  Assumption {\rm \ref{ASSUM}} is satisfied and Algorithm {\rm \ref{RDF-PRP}} generates an infinite sequence $\{X_{k}\}$. Then we have
\[
\liminf \limits_{k \to \infty} \|  F(X_{k}) \| = 0
\]
or for any accumulation point $X_{*}$  of $\{X_{k}\}$
\BE\label{thm:2}
\langle JF(X_*)[F(X_*)], F(X_*) \rangle  = 0.
\EE
\end{theorem}

\begin{proof}
Let  $X_{*}$ be any accumulation point of the sequence $\{X_{k}\}$. One may assume that
$\lim\limits_{k \to \infty} X_{k}=X_{*}$, taking a subsequence if necessary.  By Lemma \ref{lem:2} we have
\BE\label{af}
\lim \limits_{k \to \infty} \alpha_{k}^{2}f(X_{k})=\lim \limits_{k \to \infty} \frac{1}{2}\alpha_{k}^{2}\|F(X_{k})\|^{2}=0.
\EE
If $\liminf\limits_{k\to\infty}\alpha_{k}>0$, then it follows from (\ref{af}) that
\[
\liminf\limits_{k \to \infty} \|F(X_{k})\|=0.
\]
In this case, $\|F(X_{*})\|=0$ since $F$ is continuous and $\lim\limits_{k \to \infty} X_{k}=X_{*}$.

In the following, we assume that  $\liminf\limits_{k\to\infty}\alpha_{k}=0$ and $\liminf\limits_{k \to \infty} \|F(X_{k})\|>0$.
From Step 3 of Algorithm {\rm \ref{RDF-PRP}}, taking a subsequence if necessary, we may assume that $\rho^{-1}\alpha_{k}$ satisfies neither (\ref{BACKTR1}) nor (\ref{BACKTR2}) for $k$ large enough and thus
\BE\label{fx:1}
f( R_{X_k}(\rho^{-1}\alpha_k \Delta X_k) )> \Gamma_{k}+\delta_{k}-t_1\rho^{-2}\alpha_k^2\|\Delta X_k\|^{2}-t_2\rho^{-2}\alpha_k^2f(X_k)
\EE
and
\BE\label{fx:2}
f( R_{X_k}(-\rho^{-1}\alpha_k \Delta X_k) )> \Gamma_{k}+\delta_{k}-t_1\rho^{-2}\alpha_k^2\|\Delta X_k\|^{2}-t_2\rho^{-2}\alpha_k^2f(X_k).
\EE
From (\ref{PINEQUALITY}) we have $\Gamma_{k}\geq f(X_{k})\geq0$. This, together with  (\ref{fx:1}), yields
\[
f( R_{X_k}(\rho^{-1}\alpha_k \Delta X_k) )> f(X_k)-t_1\rho^{-2}\alpha_k^2\|\Delta X_k\|^{2}-t_2\rho^{-2}\alpha_k^2f(X_k).
\]
From Assumption {\rm \ref{ASSUM}}, it follows that
\[
f(X_{k})\leq f(X_{0})+\delta.
\]
By hypothesis, the sequence $\{\|F(X_k)\|\}$ is bounded from below. Thus, the condition (\ref{fr}) in  Lemma \ref{lem:4} is satisfied.
By using Lemma \ref{lem:4} and (\ref{Frx})  we have
\[
f( R_{X_k}(\rho^{-1}\alpha_k \Delta X_k) )- f(X_k)>-\Upsilon\alpha_k^2,
\]
where $\Upsilon=t_1\rho^{-2}T^2+t_2\rho^{-2}(f(X_{0})+\delta)$. Hence,
\BE\label{SPA0}
\frac{f( R_{X_k}(\rho^{-1}\alpha_k \Delta X_k) )-f(X_k)}{\alpha_k}>-\Upsilon\alpha_k.
\EE

Let $\gamma(t): = R_{X_k}(t\rho^{-1}\alpha_k \Delta X_k)$ for all  $t \in [0,1]$.
It follows from (\ref{CONJUG}) that for $t\in (0.1)$,
\begin{eqnarray}\label{GAMMACURVE}
\dot{\gamma}(t) &=& {\rm D}R_{X_k}(t\rho^{-1}\alpha_k \Delta X_k)[\rho^{-1}\alpha_k \Delta X_k]
= \rho^{-1}\alpha_k {\rm D}R_{X_k}(t\rho^{-1}\alpha_k \Delta X_k)[\Delta X_k] \nonumber\\
&=&  -\rho^{-1}\alpha_k {\rm D}R_{X_k}(t\rho^{-1}\alpha_k \Delta X_k)[F(X_k)] \nonumber\\
&& + \rho^{-1}\alpha_k {\rm D}R_{X_k}(t\rho^{-1}\alpha_k \Delta X_k)
[\ct_{\Delta Z_{k-1}}\beta_{k}\Delta X_{k-1}].
\end{eqnarray}
By the mean value theorem and using (\ref{GAMMACURVE}), there exists a $\theta\in(0,1)$ such that
\begin{eqnarray*}\label{SPA1}
& &f( R_{X_k}(\rho^{-1}\alpha_k \Delta X_k) ) - f(X_k) = f(\gamma(1)) - f(\gamma(0)) \nonumber\\
&=&  \dot{\gamma}(\theta)(f)= \langle \nabla_{\dot{\gamma}(\theta)} F, F(\gamma(\theta)) \rangle
= \langle JF(\gamma(\theta)[\dot{\gamma}(\theta)], F(\gamma(\theta)) \rangle\nonumber\\
&=& \langle JF(\gamma(\theta))\big[ \rho^{-1}\alpha_k {\rm D}R_{X_k}(\theta\rho^{-1}\alpha_k \Delta X_k)[\Delta X_k] \big], F(\gamma(\theta)) \rangle \nonumber\\
&=& \rho^{-1}\alpha_k\langle JF(\gamma(\theta))\big[  {\rm D}R_{X_k}(\theta\rho^{-1}\alpha_k \Delta X_k)[\Delta X_k] \big], F(\gamma(\theta)) \rangle \nonumber\\
&=&  -\rho^{-1}\alpha_k\langle  JF(\gamma(\theta))\big[ {\rm D}R_{X_k}(\theta\rho^{-1}\alpha_k \Delta X_k)[F(X_k)] \big], F(\gamma(\theta)) \rangle \nonumber\\
& &  + \rho^{-1}\alpha_k\langle  JF(\gamma(\theta))\big[ {\rm D}R_{X_k}(\theta\rho^{-1}\alpha_k \Delta X_k)[\ct_{\Delta Z_{k-1}}
\beta_k\Delta X_{k-1}]\big], F(\gamma(\theta)) \rangle.
\end{eqnarray*}
This, together with (\ref{SPA0}), yields
\begin{eqnarray}\label{eprp}
-\Upsilon\alpha_k &<& -\rho^{-1}\langle JF(\gamma(\theta))\big[ {\rm D}R_{X_k}(\theta\rho^{-1}\alpha_k \Delta X_k)[F(X_k)] \big],  F(\gamma(\theta)) \rangle \nonumber\\
&&+ \rho^{-1} \langle JF(\gamma(\theta))\big[ {\rm D}R_{X_k}(\theta\rho^{-1}\alpha_k \Delta X_k)[\ct_{\Delta Z_{k-1}}
\beta_k\Delta X_{k-1}]\big], F(\gamma(\theta)) \rangle.
\end{eqnarray}
By Lemma \ref{lem:2} and using the smoothness and local rigidity condition of retraction \cite[(4.2)]{AMS08} we have
\BE\label{SMRV1}
\lim_{k\to \infty}{\rm D}R_{X_k}(\theta\rho^{-1}\alpha_k \Delta X_k) = {\rm D}R_{X_*}(0_{X_*}) = {\rm id}_{T_{X_*}\cm},
\EE
where ${\rm id}_{T_{X_*}\cm}$ denotes the identity operator on $T_{X_*}\cm$.
From Lemma \ref{lem:2}, Lemma \ref{lem:4}, (\ref{RER1}), (\ref{RER2}), and using the smoothness and consistency condition of vector transport \cite[Definition 8.1.1]{AMS08}, we obtain
\BE\label{SMRV2}
\lim_{k \to \infty} \ct_{\Delta Z_{k-1}} \beta_k\Delta X_{k-1} = \ct_{0_{X_*}}0_{X_*} = 0_{X_*}.
\EE
By using  Lemma \ref{lem:2}, Lemma \ref{lem:4},  (\ref{SMRV1}), (\ref{SMRV2}), and taking limits in (\ref{eprp}), we have
\[
\langle JF(X_*)[F(X_*)], F(X_*) \rangle \leq0.
\]
Similarly, we can deduce from (\ref{fx:2}) that
\[
\langle JF(X_*)[F(X_*)], F(X_*) \rangle \geq0.
\]
The equality (\ref{thm:2}) follows from the last two inequalities.
\end{proof}

From Theorem \ref{Thm:1}, we have the following corollary.
\begin{corollary}\label{Cor}
Suppose  Assumption {\rm \ref{ASSUM}} is satisfied and  Algorithm {\rm \ref{RDF-PRP}} generates an infinite sequence $\{X_{k}\}$. Let $X_{*}$ be an accumulation point of $\{X_{k}\}$. If
\[
\langle JF(X_*)[\xi_{X_*}], \xi_{X_*} \rangle \neq0,\quad \forall 0_{X_*}\neq \xi_{X_*}\in T_{X_*}\cm,
\]
then $F(X_{*})=0_{X_*}$.
\end{corollary}

Suppose $F:\cm \to T\cm$ is a strongly geodesic monotone vector field \cite{DFL06,LLM09,NE99,WLML10}
and is continuously differentiable, then there exists a positive constant $\lambda>0$ such that
\BE\label{MONOTONE}
\langle  JF(X)[\xi_X], \xi_X \rangle > \lambda\|\xi_X\|^2, \quad \forall 0_X\neq \xi_X\in T_X\cm, \; X\in \cm.
\EE
By Corollary \ref{Cor} and (\ref{MONOTONE}), we have the following result.
\begin{corollary}\label{Cor2}
Suppose $F$ or $-F$ is strictly monotone and continuously differentiable, and  Algorithm {\rm \ref{RDF-PRP}} generates an infinite sequence $\{X_{k}\}$. Then every accumulation point of $\{X_k\}$ is a zero of $F$.
\end{corollary}

\section{Numerical Experiments}\label{sec4}
\setcounter{table}{0}
\setcounter{figure}{0}

In this section, we consider the application of Algorithm \ref{RDF-PRP} to finding zeros of Oja's vector fields \cite{AILH09}, the tangent vector field corresponding to the trace ratio optimization problem {\rm \cite{NB10, ZL14, ZL15}}, and monotone tangent vector fields on Hadamard manifolds \cite{DFLS06}.
All numerical tests are carried out using {\tt MATLAB} R2010a on a Lenovo Laptop Intel(R) Core(TM)2 i7-8550U with a 1.80 GHz CPU and 16-GB RAM.

In our numerical tests, we set $\rho=0.5$, $\lambda_k = 0.6$, $t_1 = t_2 = 10^{-10}$, $\alpha_{\min} = 10^{-10}$,
$\alpha_{\max} = 10^{10}$, and $\delta_k = \|F(X_0)\|/((2+k)\ln^2(2+k))$ for all $k$.
In Step 3 of Algorithm \ref{RDF-PRP}, the initial steplength $\alpha_{k_0}$ is set to be
\[
\alpha_{k_0} =
\left\{
\begin{array}{ll}
\alpha_{\max}, &\quad \mbox{if}\; \sigma > \alpha_{\max}, \\[2mm]
\sigma,        &\quad \mbox{if}\; \sigma \in [\alpha_{\min}, \alpha_{\max}], \\[2mm]
\alpha_{\min}, &\quad \mbox{if}\; \sigma < \alpha_{\max},
\end{array}
\right.
\]
where
\[
\sigma = \left | \frac{\langle F(X_k), \Delta X_k \rangle}{\langle Z_k, \ct_{\epsilon \Delta X_k}\Delta X_k \rangle} \right |,
\quad Z_k = \frac{F(R_{X_k}(\epsilon \Delta X_k))- \ct_{\epsilon \Delta X_k}F(X_k)}{\epsilon}, \quad \epsilon = 10^{-8}.
\]
The stopping criterion for Algorithm \ref{RDF-PRP} for solving (\ref{prb}) is set to be \cite{CMR06,ML14}
\[
\frac{\|F(X_k)\|}{\sqrt{M}} \leq e_a + e_r \frac{\|F(X_0)\|}{\sqrt{M}},
\]
where $e_a=10^{-6}$, $e_r = 10^{-5}$, and $M$ denotes the dimension of $\cm$.

For comparison purposes,  we repeat our experiments over $10$ different random generated problems.
In our numerical tests,  `{\tt DIM.}' denotes the dimension of  $\cm$,
`{\tt CT.}', {\tt IT.}', and `{\tt NF.}' mean the averaged total computing time in seconds,
the averaged number of iterations, the averaged number of function evaluations
at the final iterates of our algorithm accordingly.
In addition,  `{\tt Res0.}' and `{\tt Res.}' denote the averaged residual $\|F(X_k)\|$
at the initial iterates and final iterates of our algorithm, respectively.

\begin{example}\label{ex:1}
We consider the problem of finding a zero of Oja's vector field defined by real symmetric positive-definite matrices {\rm \cite{AILH09}}.
Let $A\in \R^{m\times m}$ be a symmetric positive-definite matrix, and $p$ be a positive integer smaller than $m$.
The Oja's vector field $F:\R^{m\times p} \to \R^{m\times p}$ associated with $A$ is given by {\rm \cite{AILH09,OJA82,OJA89}}
\BE\label{OJASTIEFEL}
F(X)= AX - XX^TAX,\quad \forall X\in\R^{m\times p}.
\EE
Suppose $X\in \R^{m\times p}$ is of full column rank. Then $X$ is a solution to $F(X)= \mathbf{0}$ if and only if
the column space of $X$ is an invariant subspace of $A$ and $X$ is orthonormal
{\rm (}i.e., $X^TX = I_p${\rm )} {\rm (}see {\rm \cite[Proposition 2.1]{AILH09})},
where $I_p$ is the identity matrix of order $p$.
Thus we can restrict the nonlinear map $F$ to the compact Stiefel manifold
${\rm St}(p,m)$ {\rm \cite[p.26]{AMS08}}, i.e., $F:{\rm St}(p,m) \to T{\rm St}(p,m)$.
The dimension of the Stiefel manifold ${\rm St}(p,m)$ is equal to $mp - \frac{1}{2}p(p+1)$ {\rm \cite[p.27]{AMS08}}.
Let $\mathcal{O}(p)= {\rm St}(p,p)$, which is the orthogonal group {\rm \cite[p.27]{AMS08}}.
Since $F(XQ) =F(X)Q$ for any $Q\in \mathcal{O}(p)$, the zeros of $F$ are degenerate, thus Newton's method can't be applied directly.
To apply Riemannian Newton's method, one need to restrict $F$ to the Grassmann manifold ${\rm Grass}(p,m):= {\rm St}(p,m)/\mathcal{O}(p)$,
while the application of Algorithm {\rm \ref{RDF-PRP}} to finding a zero of {\rm (\ref{OJASTIEFEL})} does not need
the nondegeneracy condition of the zeros of $F$.
Let ${\rm St}(p,m)$ be endowed with induced Riemannian metric from $\R^{m\times p}$, i.e.,
\[
g_X(\xi_X,\eta_X):=\tr(\xi_X^T\eta_X),\quad\forall \xi_X,\eta_X\in T_X \st(p,m),\; X\in\st(k,n).
\]
The retraction $R$ on $\st(k,n)$ is chosen as {\rm \cite[p.59]{AMS08}}
\BE\label{retr:st}
R_X(\xi_X) = \qf(X + \xi_X),
\EE
for all $\xi_X\in T_X \st(p,m)$ and $X\in\st(p,m)$,
where $\qf(X+\xi_X)$ is the $Q$ factor of the QR decomposition of $X+\xi_X\in\R^{m\times p}_*$ with
$X+\xi_X = Q\widetilde{R}$.
Here, the set $\R^{m\times p}_*$ denotes the set of all real $m\times p$ matrices with linearly independent columns,
$Q\in\st(p,m)$, and $\widetilde{R}$ is an upper triangular $p\times p$ matrix with strictly positive diagonal elements.
The orthogonal projection of a matrix $Z\in \R^{m\times p}$ onto $T_X \st(p, m)$ is given by
\[
\mathrm{P}_XZ = (I_n - XX^T )Z + X \skw(X^TZ)=Z-X\sym(X^TZ),
\]
where $\skw(A):=(A-A^T)/2$ and  $\sym(A):=(A+A^T)/2$ for a real square matrix. Since $\st(p,m)$ is an embeded submanifold of $\R^{m\times p}$, we may adopt the vector transport defined by {\rm \cite[p.174]{AMS08}}
\BE\label{TTST}
\ct_{\eta_X}\xi_X := (I_n - YY^T )\xi_X + Y \skw(Y^TZ)= \xi_X-Y\sym(Y^T\xi_X),
\EE
for $\xi_X, \eta_X\in T_X \st(p,m)$, where $Y:= R_X(\eta_X)\in \st(p,m)$.
Thus condition {\rm {\rm (\ref{TRANSPORT})}}
in Assumption {\rm {\rm \ref{ASSUM}}} is satisfied.

We consider the problem of finding a zero of the Oja's vector field $F:{\rm St}(p,m) \to T{\rm St}(p,m)$ defined by {\rm (\ref{OJASTIEFEL})} with varying $m$ and $p$.
Let $A$ be a random $m\times m$ matrix generated by the {\tt MATLAB} built-in functions {\tt rand}, {\tt randn}, and {\tt qr}{\rm :}
\[
D = \mbox{\tt rand}(m,1), \quad  B = \mbox{\tt randn}(m,m),
\quad [Q,S] = \mbox{\tt qr}(B), \quad A = QDQ^{T} .
\]
Thus $A$ is a random symmetric positive-definite matrix with uniformly distributed eigenvalues in the interval $[0,1]$.
The starting points are randomly generated by the {\tt MATLAB} built-in functions {\tt randn} and {\tt qr}{\rm :}
\[
\begin{array}{lcl}
W = \mbox{\tt randn}\,(m,p),& \quad & \big[ X_0, \widehat{R} \big] = \mbox{\tt qr}\,(W).
\end{array}
\]
\end{example}

Table \ref{table1} lists the numerical results for Example \ref{ex:1}.
We observe from Table \ref{table1} that the iteration number and the number of function evaluations do not change obviously with the increase of the dimension of the Stiefel manifold $\st(p,m)$. This indicates that Algorithm \ref{RDF-PRP} is stable and suitable for solving large-scale problems.


To further illustrate the effectiveness of our algorithm, in Figure \ref{f1}, we give the convergence history of Algorithm \ref{RDF-PRP} for two tests with $(m,p)=(6000,30)$ and $(m,p)=(3000,120)$.
Figure \ref{f1} depicts the logarithm of the residual versus the number of iterations for finding a zero of Oja's vector field defined in  Example \ref{ex:1}.
The convergence trajectory indicates that the residual decreases steadily as the number of iterations increases.
\begin{table}[htbp]
  \caption{Numerical results for Example \ref{ex:1}.}\label{table1}
  \begin{center} {\scriptsize
   \begin{tabular}[c]{|r|r|r|c|c|c|c|}
     \hline
     \multicolumn{7}{|c|}{$p=30$}\\ \hline
  $m$ & {\tt DIM.} & {\tt CT.}  & {\tt IT.} & {\tt NF.} & {\tt Res0.} & {\tt Res.} \\  \hline
 1000   &  29535    &  0.6938 s   & 131.7   &  137.7  &   1.5558  & $1.8068\times 10^{-4}$   \\  \hline
 2000   &  59535    &  2.5411 s   & 147.5   &  154.1  &   1.5714  & $2.5403\times 10^{-4}$   \\  \hline
 3000   &  89535    &  5.1273 s   & 179.4   &  185    &   1.5780  & $3.1178\times 10^{-4}$   \\  \hline
 4000   & 119535    &  7.8231 s   & 176.3   &  185.7  &   1.5746  & $3.5450\times 10^{-4}$     \\  \hline
 5000   & 149535    &  14.0162 s  & 186.9   &  195.9  &   1.5710  & $3.9744\times 10^{-4}$     \\  \hline
 6000   & 179535    &  19.1962 s  & 188.3   &  198.1  &   1.5707  & $4.3520\times 10^{-4}$     \\  \hline
 7000   & 209535    &  25.5629 s  & 179.4   &  189.8  &   1.5765  & $4.6722\times 10^{-4}$   \\  \hline
 8000   & 239535    &  33.7059 s  & 176.3   &  197.1  &   1.5816  & $5.0002\times 10^{-4}$     \\  \hline
 9000   & 269535    &  41.5631 s  & 186.9   &  194.1  &   1.5789  & $5.3148\times 10^{-4}$     \\  \hline
 10000  & 299535    &  54.4226 s  & 188.3   &  202.7  &   1.5754  & $5.5995\times 10^{-4}$     \\  \hline
   \cline{1-5} \hline
     \hline
     \multicolumn{7}{|c|}{$m=3000$}\\ \hline
  $p$ & {\tt DIM.} & {\tt CT.}  & {\tt IT.} & {\tt NF.} & {\tt Res0.} & {\tt Res.} \\  \hline
 20   &  59790  &  3.9187 s   & 186.1  &  193.7  &  1.2880  & $2.5059\times 10^{-4}$  \\  \hline
 40   & 119180  &  5.7383 s   & 169.0  &  175.2  &  1.8110  & $3.5750\times 10^{-4}$  \\  \hline
 60   & 178170  &  9.1466 s   & 170.9  &  179.1  &  2.2136  & $4.3705\times 10^{-4}$  \\  \hline
 80   & 236760  & 11.8402 s   & 165.5  &  176.1  &  2.5571  & $5.0355\times 10^{-4}$  \\  \hline
 100  & 294950  & 15.9621 s   & 155.2  &  165.6  &  2.8397  & $5.6216\times 10^{-4}$  \\  \hline
 120  & 352740  & 17.1358 s   & 151.5  &  160.9  &  3.1074  & $6.1428\times 10^{-4}$  \\  \hline
 140  & 410130  & 19.2442 s   & 139.3  &  148.9  &  3.3493  & $6.6242\times 10^{-4}$  \\  \hline
 160  & 467120  & 21.7834 s   & 134.7  &  143.9  &  3.5464  & $7.0720\times 10^{-4}$  \\  \hline
 180  & 523710  & 24.5269 s   & 131.8  &  142.4  &  3.7413  & $7.4701\times 10^{-4}$  \\  \hline
 200  & 579900  & 30.2070 s   & 157.5  &  149.1  &  3.9330  & $7.8585\times 10^{-4}$  \\  \hline
   \cline{1-5} \hline
  \end{tabular} }
  \end{center}
\end{table}
\begin{figure}[htbp]
\begin{center}
\begin{tabular}{cc}
\epsfig{figure=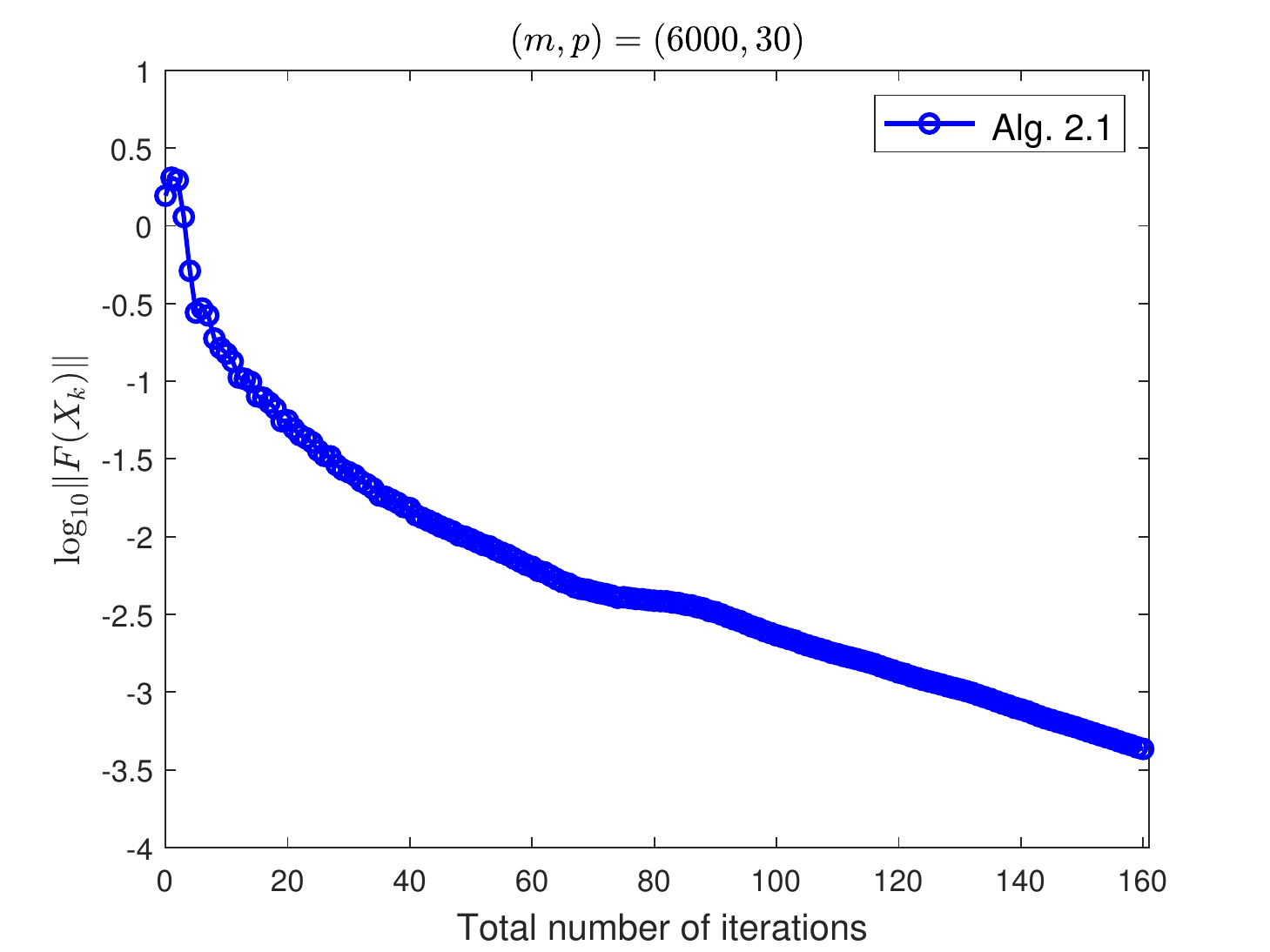,width=6.2cm} & \epsfig{figure=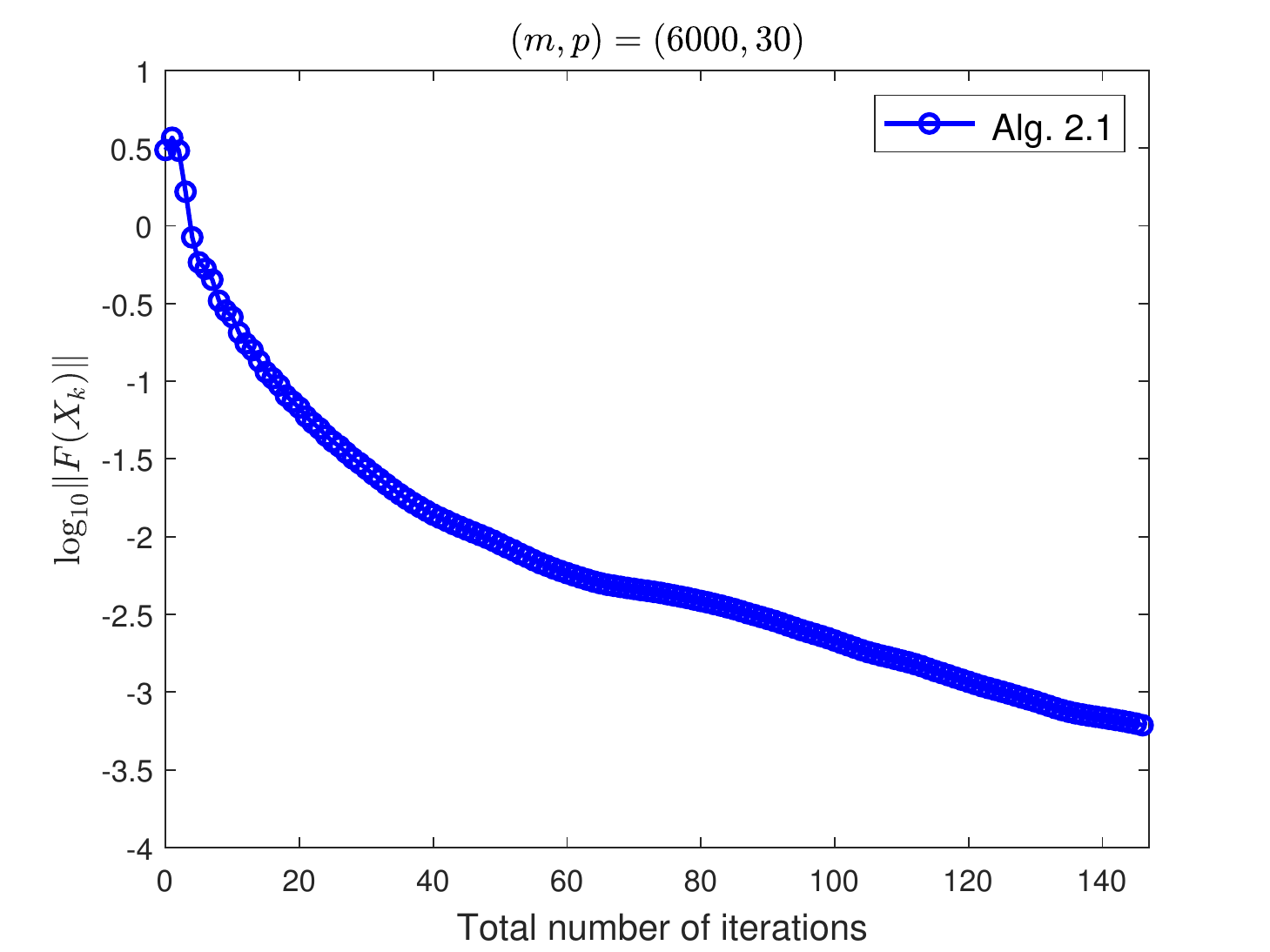,width=6.2cm}
\end{tabular}
\end{center}
 \caption{Convergence history of two tests for Example \ref{ex:1}.}
 \label{f1}
\end{figure}

\begin{example}\label{ex:2}
We consider the problem of finding a zero of the tangent vector field corresponding to the first-order optimization conditions for
the trace ratio optimization problem {\rm \cite{NB10, ZL14, ZL15}}.
Let $A,B,C\in \R^{m\times m}$ be real symmetric matrices with $B$ being positive-definite
and $p$ be a positive integer smaller than $m/2$.
The tangent vector field $F:{\rm St}(p,m) \to T{\rm St}(p,m)$ is given by {\rm \cite[Theorem 2.1]{ZL14}}
\BE\label{TRACERATIO}
F(X)= E(X)X-X(X^TE(X)X),\quad\forall X\in\st(p,m),
\EE
where
\[
E(X):= A\frac{1}{\phi_B(X)} - B\frac{\phi_A(X)}{\phi_B^2(X)} + C,
\]
and $\phi_S(X):={\rm tr}(X^TSX)$ for any $m\times m$ real symmetric matrix $S$. We choose the retraction $R$ on $\st(k,n)$  as in {\rm (\ref{retr:st})}. The vector transport on $\st(p,m)$ is chosen the same as {\rm (\ref{TTST})} and thus condition {\rm {\rm (\ref{TRANSPORT})}}
in Assumption {\rm {\rm \ref{ASSUM}}} is satisfied.

We consider the problem of finding a zero of the tangent vector field $F$ defined by
{\rm (\ref{TRACERATIO})} with varying $m$ and $p$.
Let $A,B,C$ be random $m\times m$ matrices generated by the {\tt MATLAB}
built-in functions {\tt rand}, {\tt randn}, {\tt orth}, {\tt diag}, and {\tt ones} {\rm \cite{CJB18}:}
\[
\begin{array}{l}
A = \mbox{\tt rand}(m,m), \quad  A = (A+A^T)/2, \quad Q = {\tt orth}({\tt randn}(m,m)),\\[2mm]
B = Q\mbox{\tt diag}(50+10*(2*{\rm rand}(m,1)-{\tt ones}(m,1)))*Q^T, \quad B = (B+B^T)/2, \\[2mm]
C = \mbox{\tt randn}(m,m), \quad  C = (C+C^T)/2.
\end{array}
\]
The starting points are randomly generated by the {\tt MATLAB} built-in functions {\tt randn} and {\tt qr}{\rm :}
\[
\begin{array}{lcl}
W = \mbox{\tt randn}\,(m,p),& \quad & \big[ X_0, \widehat{R} \big] = \mbox{\tt qr}\,(W).
\end{array}
\]
\end{example}

In Table \ref{table3}, we report numerical results for Example \ref{ex:2} with varying values of $m$ and $p$. In Figure \ref{f2}, we give the convergence history of Algorithm \ref{RDF-PRP} for two tests with $(m,p)=(3000,30)$ and $(m,p)=(2000,100)$. Figure \ref{f2} depicts the logarithm of the residual versus the number of iterations for finding a zero of the tangent vector field $F$ defined in (\ref{TRACERATIO}). We see from Table \ref{table3} and Figure \ref{f2} that Algorithm \ref{RDF-PRP} is stable and efficient for solving large-scale problems.
\begin{table}[htbp]
  \caption{Numerical results for Example \ref{ex:2}.}\label{table3}
  \begin{center} {\scriptsize
   \begin{tabular}[c]{|r|r|r|c|c|c|c|}
     \hline
     \multicolumn{7}{|c|}{$p=30$}\\ \hline
  $m$ & {\tt DIM.} & {\tt CT.}  & {\tt IT.} & {\tt NF.} & {\tt Res0.} & {\tt Res.} \\  \hline
  200 &   5535  &  0.1362  s   & 100.6  &  112.4 &  $5.0376\times 10^{1}$  & $5.3678\times 10^{-4}$ \\  \hline
  400 &  11535  &  0.4595  s   & 114.6  &  128.2 &  $7.4573\times 10^{1}$  & $8.1477\times 10^{-4}$ \\  \hline
  600 &  17535  &  1.5889  s   & 135.5  &  149.1 &  $9.2751\times 10^{1}$  & $1.0105\times 10^{-3}$ \\  \hline
  800 &  23535  &  1.5889  s   & 124.8  &  138.6 &  $1.0726\times 10^{2}$  & $1.1774\times 10^{-3}$ \\  \hline
 1000 &  29535  &  2.6262  s   & 139.3  &  156.9 &  $1.2027\times 10^{2}$  & $1.3185\times 10^{-3}$ \\  \hline
 2000 &  59535  & 13.6532  s   & 219.3  &  235.5 &  $1.7193\times 10^{2}$  & $1.9084\times 10^{-3}$ \\  \hline
 3000 &  89535  & 36.7627  s   & 276.1  &  291.9 &  $2.1109\times 10^{2}$  & $2.3692\times 10^{-3}$ \\  \hline
 4000 & 119535  & 63.2379  s   & 275.2  &  292.6 &  $2.4399\times 10^{2}$  & $2.7274\times 10^{-3}$ \\  \hline
 5000 & 149535  & 120.2924 s   & 307.0  &  325.6 &  $2.7326\times 10^{2}$  & $3.0436\times 10^{-3}$ \\  \hline
   \cline{1-5} \hline
%
     \hline
     \multicolumn{7}{|c|}{$m=2000$}\\ \hline
  $p$ & {\tt DIM.} & {\tt CT.}  & {\tt IT.} & {\tt NF.} & {\tt Res0.} & {\tt Res.} \\  \hline
 20  &  39790 &  8.6164 s   & 170.7  &  186.1 & $1.4107\times 10^{2}$  & $1.5714\times 10^{-3}$ \\  \hline
 40  &  79180 & 13.0390 s   & 196.6  &  212.8 & $1.9777\times 10^{2}$  & $2.2190\times 10^{-3}$  \\  \hline
 60  & 118170 & 17.9773 s   & 211.5  &  228.9 & $2.4134\times 10^{2}$  & $2.7032\times 10^{-3}$  \\  \hline
 80  & 156760 & 22.1530 s   & 198.1  &  215.3 & $2.7716\times 10^{2}$  & $3.1042\times 10^{-3}$  \\  \hline
 100 & 194950 & 31.0476 s   & 241.0  &  258.6 & $3.0827\times 10^{2}$  & $3.4166\times 10^{-3}$  \\  \hline
 120 & 232740 & 38.5245 s   & 238.3  &  258.7 & $3.3572\times 10^{2}$  & $3.7598\times 10^{-3}$  \\  \hline
 140 & 270130 & 42.8672 s   & 232.4  &  249.8 & $3.6124\times 10^{2}$  & $4.0373\times 10^{-3}$  \\  \hline
 160 & 307120 & 53.2935 s   & 252.0  &  270.6 & $3.8352\times 10^{2}$  & $4.2983\times 10^{-3}$  \\  \hline
 180 & 343710 & 48.4526 s   & 203.9  &  223.1 & $4.0462\times 10^{2}$  & $4.5413\times 10^{-3}$  \\  \hline
 200 & 379900 & 60.0266 s   & 222.1  &  240.3 & $4.2432\times 10^{2}$  & $4.7507\times 10^{-3}$  \\  \hline
   \cline{1-5} \hline
  \end{tabular} }
  \end{center}
\end{table}
\begin{figure}[htbp]
\begin{center}
\begin{tabular}{cc}
\epsfig{figure=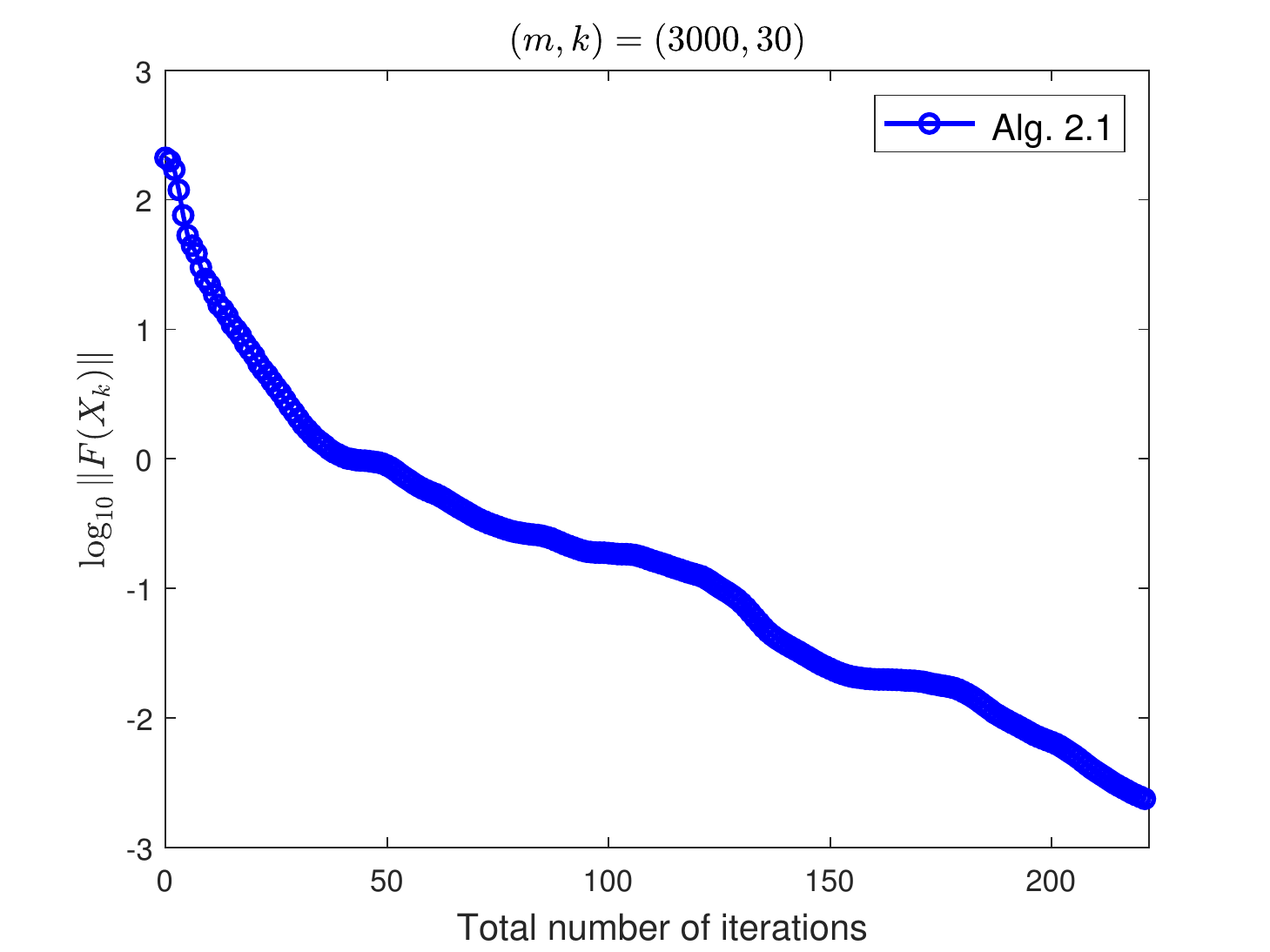,width=6.2cm} & \epsfig{figure=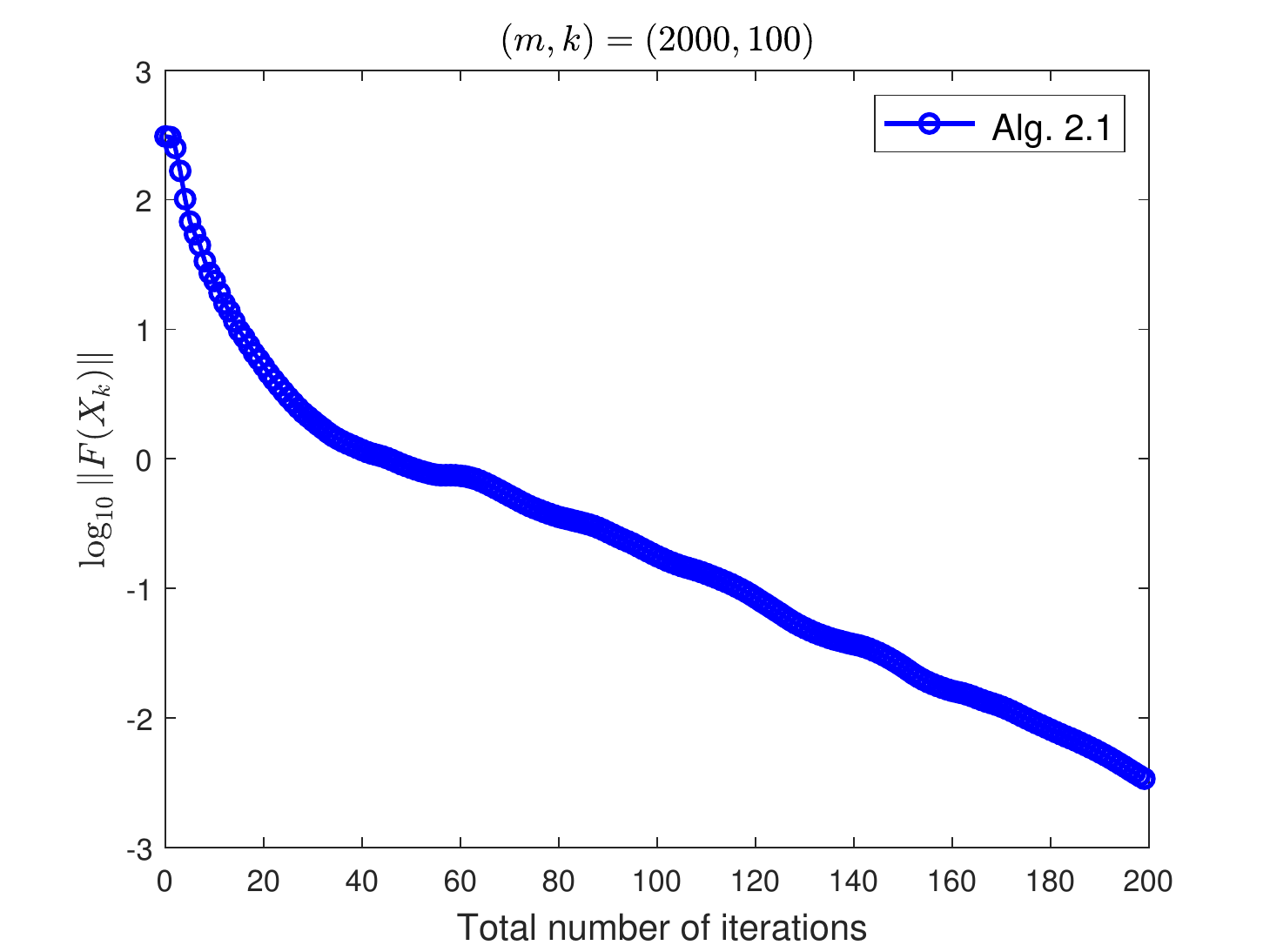,width=6.2cm}
\end{tabular}
\end{center}
 \caption{Convergence history of two tests for Example \ref{ex:2}.}
 \label{f2}
\end{figure}

\begin{example}\label{ex:3}
Let $S_{++}^m$ denote the set of all $m\times m$ real symmetric positive definite matrices.
Endowing $S_{++}^m$ with the following Riemannian metric
\[
\langle \xi_X, \eta_X \rangle := {\rm tr}(\xi_X X^{-1} \eta_XX^{-1}), \quad \forall \xi_X,\eta_X\in T_X S_{++}^m, \; X\in S_{++}^m.
\]
Thus, $S_{++}^m$ is a Hadamard manifold manifold of nonpositive curvature everywhere {\rm \cite{JE15,WLLY16}}.
The dimension of $S_{++}^m$ is equal to $m(m+1)/2$ {\rm \cite[Proposition 2.1]{HS95}}. The geodesic monotone vector field $F: S_{++}^m \to TS_{++}^m$ is defined by {\rm \cite{DFLS06}}
\BE\label{MONVECTOR}
F(X)=2( {\rm ln}\det(X))X,\quad\forall X\in S_{++}^m.
\EE
The retraction $R$ on $S_{++}^m$ is chosen as {\rm \cite[(3.10)]{JE15}}
\[
R_X(\xi_X) = \xi_X,
\]
for $\xi_X\in T_X S_{++}^m$ and $X\in S_{++}^m$.
The vector transport associated with the above  $R$ is chosen as {\rm \cite[(3.13)]{JE15}}
\[
\ct_{\eta_X}\xi_X = \xi_X,
\]
for $\xi_X,\eta_X\in T_X S_{++}^m$ and $X\in S_{++}^m$. Thus condition {\rm {\rm (\ref{TRANSPORT})}}
in Assumption {\rm {\rm \ref{ASSUM}}} is satisfied.

We consider the problem of finding a zero of the vector field $F$ defined by
{\rm (\ref{MONVECTOR})} with varying $m$.
The starting points are randomly generated by the {\tt MATLAB} built-in functions
{\tt rand}, {\tt randn}, and {\tt qr}{\rm :}
\[
G = 0.1 + \mbox{\tt rand}(m,1), \quad  H = \mbox{\tt randn}(m,m), \quad [W,T] = \mbox{\tt qr}(H), \quad X_0 = WGW^{T} .
\]
\end{example}

Table \ref{table5} shows the numerical results for Example \ref{ex:3}.
We observe from Table \ref{table5} that Algorithm \ref{RDF-PRP} requires only a few iterations and function evaluations
for finding an approximate zero of the monotone vector field (\ref{MONVECTOR}) with different values of $m$.
This indicates that Algorithm \ref{RDF-PRP} is very stable and efficient for solving large-scale problems.
In Figure \ref{f3}, we give the convergence history of Algorithm \ref{RDF-PRP} for two tests with $m=600$ and $m=1000$.
Figure \ref{f3} depicts the logarithm of the residual versus the number of iterations for finding a zero of the tangent vector field $F$ defined in (\ref{MONVECTOR}).
The convergence trajectory  indicates that the residual decreases very rapidly as the number of iterations increases,
which shows the local fast convergence speed of Algorithm \ref{RDF-PRP} for solving large-scale problems.
\begin{table}[htbp]
  \caption{Numerical results for Example \ref{ex:3}.}\label{table5}
  \begin{center} {\small
   \begin{tabular}[c]{|r|r|r|c|c|c|c|}
     \hline
  $m$ & {\tt DIM.} & {\tt CT.}  & {\tt IT.} & {\tt NF.} & {\tt Res0.} & {\tt Res.} \\  \hline
 100   &  5050   &  0.0159 s  & 5.9   &  7.0  &   $1.3313\times 10^{3}$  & $2.0499\times 10^{-4}$   \\  \hline
 200   &  20100  &  0.0432 s  & 6.2   &  7.2  &   $3.7973\times 10^{3}$  & $3.8884\times 10^{-4}$   \\  \hline
 300   &  45150  &  0.1038 s  & 6.4   &  7.4  &   $6.9051\times 10^{3}$  & $3.0280\times 10^{-5}$   \\  \hline
 400   &  80200  &  0.2603 s  & 6.5   &  7.5  &   $1.0638\times 10^{4}$  & $6.7415\times 10^{-5}$   \\  \hline
 500   & 125250  &  0.4486 s  & 6.6   &  7.6  &   $1.4760\times 10^{4}$  & $9.8202\times 10^{-5}$   \\  \hline
 600   & 180300  &  0.6918 s  & 6.3   &  7.3  &   $1.9674\times 10^{4}$  & $6.4741\times 10^{-5}$   \\  \hline
 700   &  245350 &  1.0513 s  & 6.4   &  7.4  &   $2.4651\times 10^{4}$  & $1.7830\times 10^{-4}$   \\  \hline
 800   &  320400 &  1.5852 s  & 6.6   &  7.6  &   $3.0098\times 10^{4}$  & $2.6627\times 10^{-4}$   \\  \hline
 900   & 405450  &  2.1248 s  & 6.6   &  7.6  &   $3.6046\times 10^{4}$  & $3.6220\times 10^{-4}$   \\  \hline
 1000  & 500500  &  2.8110 s  & 6.5   &  7.5  &   $4.2361\times 10^{4}$  & $2.6249\times 10^{-4}$   \\  \hline
   \cline{1-5} \hline
  \end{tabular} }
  \end{center}
\end{table}
\begin{figure}[htbp]
\begin{center}
\begin{tabular}{cc}
\epsfig{figure=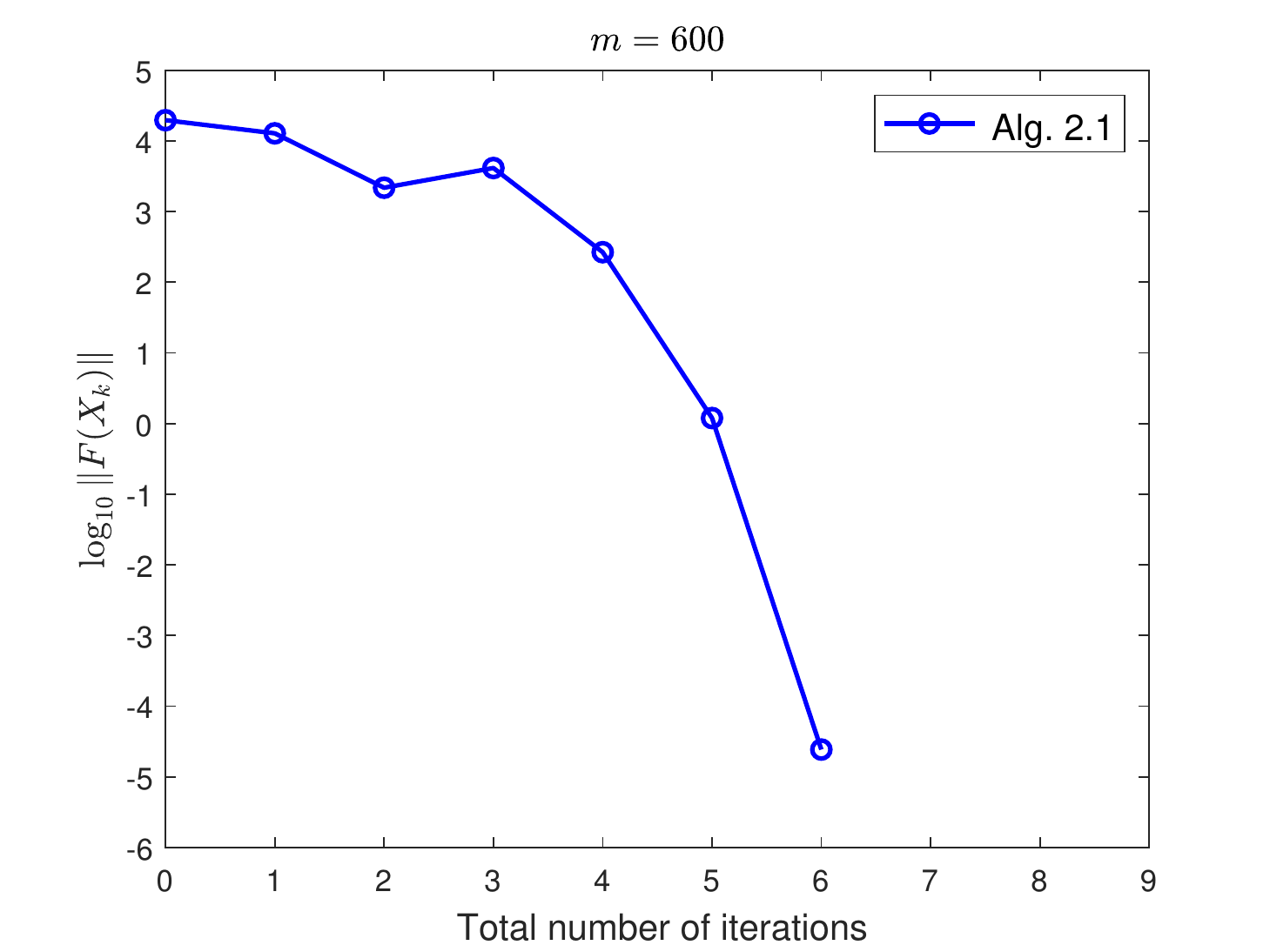,width=6.2cm} & \epsfig{figure=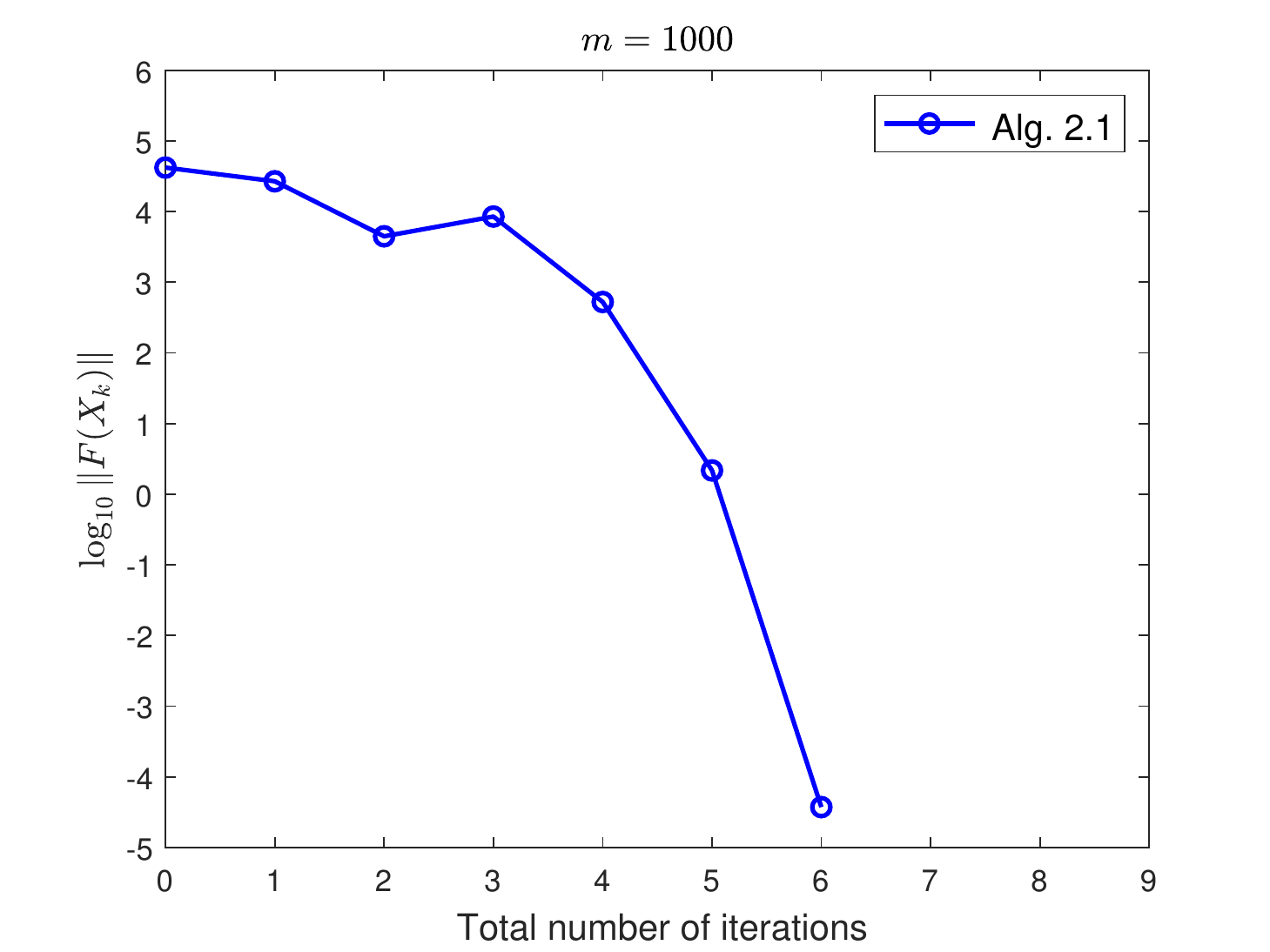,width=6.2cm}
\end{tabular}
\end{center}
 \caption{Convergence history of two tests for Example \ref{ex:3}.}
 \label{f3}
\end{figure}

\section{Hybrid Method}\label{sec5}
\setcounter{table}{0}
\setcounter{figure}{0}

We note that Algorithm \ref{RDF-PRP} is globally convergent. We see from the numerical experiments in section \ref{sec4} that, in general, Algorithm \ref{RDF-PRP} converges at a low or medium  order of accuracy. To improve the efficiency,  one may adopt some hybrid method. A possible strategy is to combine Algorithm \ref{RDF-PRP} with the Riemannian Newton method. As noted in section \ref{sec1},
the Riemannian Newton method may be computationally expensive but has quadratic convergence. In particular, one may use Algorithm \ref{RDF-PRP} to generate an initial point for the Riemannian Newton method with a relatively low accuracy and then switch to the Riemannian Newton method for finding a solution of  high accuracy.
A hybrid algorithm for solving (\ref{prb}) is described as follows.
\begin{algorithm}\label{PRP-Newton}
{\rm (PRP-Newton Method)}
\begin{description}
\item [{\rm Step 0.}]
Choose an initial point $ X_{0} \in \mathcal{M}$,  $0<\zeta_2 < \zeta_1$, and $0<\varsigma<1$, $t_1,t_2>0$, $0<\rho<1$, $0<\lambda_{\min}<\lambda_{\max}<1$, $0<\alpha_{\min}\le\alpha\le\alpha_{\max}$. Let $k:=0$, $\Gamma_{0}:=f(X_{0})$, $\Phi_{0}:=1$.
Select a positive sequence $\{\delta_{k}\}$ such that {\rm (\ref{n})} is satisfied.

\item [{\rm Step 1.}] For $k=1,2,\ldots$, do the RDF-PRP iteration as follows{\rm :}
\begin{itemize}
 \item[{\rm (a).}] Set $\Delta X_k$ to be {\rm (\ref{CONJUG})} where $\beta_k$ and $Y_k$ are given by {\rm (\ref{def:betak})}.
 \item[{\rm (b).}]   Determine $\alpha_k = \max \{ \alpha\rho^j, j = 0,1,2,\ldots \}$ such that if the condition {\rm (\ref{BACKTR1})} is satisfied,
then compute $X_{k+1}$ from {\rm  (\ref{RER1})}; else if the condition {\rm (\ref{BACKTR2})} is satisfied,
then compute $X_{k+1}$ from {\rm  (\ref{RER2})}.
 \item[{\rm (c).}]  Choose $\lambda_{k}\in[\lambda_{\min},\lambda_{\max}]$ and compute $\Phi_{k+1}=\lambda_{k}\Phi_{k}+1$ and $\Gamma_{k+1}$ from  {\rm  (\ref{QC})}.
  \item[{\rm (d).}] Stop if $\|F(X_k)\| < \zeta_1$.
\end{itemize}
\item [{\rm Step 2.}] Set $X_0$ to be the limit point of the RDF-PRP iteration.

\item [{\rm Step 3.}] For $k=1,2,\ldots$, do the Riemannian Newton iteration as follows{\rm :}
\begin{itemize}
\item[{\rm (a).}] Apply the conjugate gradient {\rm (CG)} method {\rm \cite[Algorithm 10.2.1]{GV96}} to solving
               \[
               JF(X_{k-1})[ \Delta X_{k-1} ] = -F(X_{k-1})
               \]
               for $\Delta X_{k-1}\in T_{X_{k-1}}\cm$ such that
               \[
               \|JF(X_{k-1})[ \Delta X_{k-1} ] + F(X_{k-1})\| \leq \varsigma_{k-1} \|F(X_{k-1})\|,
               \]
where $\varsigma_{k-1}:=\min\{\varsigma,\|F(X_{k-1})\|\}$.
\item[{\rm (b).}] Set
               \[
                X_{k} := R_{X_{k-1}}(\Delta X_{k-1}).
               \]

\item[{\rm (c).}] Stop if $\|F(X_k)\| < \zeta_2$.
\end{itemize}
\end{description}
\end{algorithm}

We point out that, in Step 3 of Algorithm \ref{PRP-Newton}, the Riemannian Newton equation is solved inexactly by choosing appropriate value of $\varsigma$.
In addition, different values of $\zeta_1$ lead to different starting points for the Riemannian Newton method.

For demonstration purpose, we use Algorithm \ref{PRP-Newton} to Examples \ref{ex:1}--\ref{ex:2}, i.e., finding zeros of the tangent vector fields defined by (\ref{OJASTIEFEL}) and (\ref{TRACERATIO}).
To develop the Riemannian Newton method, one need to restrict the tangent vector fields in (\ref{OJASTIEFEL}) and (\ref{TRACERATIO}) to the Grassmann manifold ${\rm Grass}(p,m)$ endowed with  the induced Riemannian metric from ${\rm St}(p,m)$.
The restriction $\widehat{F}:{\rm Grass}(p,m) \to T{\rm Grass}(p,m)$ of $F$ defined in (\ref{OJASTIEFEL}) to ${\rm Grass}(p,m)$ is given by
\[
\widehat{F}([X] )=[AX - XX^TAX],
\]
where $[X]:=\{XQ\in\st(p,m) \ | \ Q\in \mathcal{O}(p)\}\in {\rm Grass}(p,m)$ denotes the equivalent class corresponding to a point $X\in {\rm St}(p,m)$.
Given $X\in {\rm St}(p,m)$ and a tangent vector $\xi_{[X]}\in T_{[X]}{\rm Grass}(p,m)$,
let $\overline{\xi_{[X]}}\in \ch_X$ denote the horizontal lift of $\xi_{[X]}\in T_{[X]}{\rm Grass}(p,m)$
at $X\in {\rm St}(p,m)$, where $\ch_X$ denotes the horizontal space at $X\in {\rm St}(p,m)$ \cite[p.757]{ZBJ15}.
The horizontal lift of $J\widehat{F}([X])\big[\xi_{[X]}\big]\in \ch_X$ at $X\in {\rm St}(p,m)$ is denoted by
$\overline{J\widehat{F}([X])\big[\xi_{[X]}\big]}$, which has the following form:
\[
\overline{J\widehat{F}([X])\big[\xi_{[X]}\big]} = (I-XX^T)(A\overline{\xi_{[X]}}-\overline{\xi_{[X]}}X^TAX).
\]

Similarly, the restriction $\widehat{F}:{\rm Grass}(p,m) \to T{\rm Grass}(p,m)$ of $F$ defined in (\ref{TRACERATIO}) to ${\rm Grass}(p,m)$ is given by
\[
\widehat{F}([X])= [E(X)X-X(X^TE(X)X)],\quad \forall [X]\in {\rm Grass}(p,m).
\]
Given a point $X\in {\rm St}(p,m)$ and a tangent vector $\xi_{[X]}\in T_{[X]}{\rm Grass}(p,m)$,
the horizontal lift of $J\widehat{F}([X])\big[\xi_{[X]}\big]\in \ch_X$
at $X\in {\rm St}(p,m)$ is denoted by $\overline{J\widehat{F}([X])\big[\xi_{[X]}\big]}$,
which is given by
\[
\overline{J\widehat{F}([X])\big[\xi_{[X]}\big]} =
(I-XX^T)\big(E(X)\overline{\xi_{[X]}} + G(X,\overline{\xi_{[X]}})X -\overline{\xi_{[X]}}X^TE(X)X\big),
\]
where
\[
G(X,\overline{\xi_{[X]}}):=  A\frac{-\phi'_B(X;\overline{\xi_{[X]}})}{\phi^2_B(X)}
- B\frac{\phi'_A(X;\overline{\xi_{[X]}})\phi_B^2(X)- 2\phi_B(X)\phi_B'(X;\overline{\xi_{[X]}})\phi_A(X)}{\phi_B^4(X)}
\]
and
\[
\phi'_A(X;\overline{\xi_{[X]}}) := 2{\rm tr}(V^TA\overline{\xi_{[X]}}), \qquad
\phi'_B(X;\overline{\xi_{[X]}}) := 2{\rm tr}(V^TB\overline{\xi_{[X]}}).
\]
For the application of Riemannian optimization algorithms on Riemannian quotient manifolds,
one can refer to \cite[p.86 and p.121]{AMS08} and \cite{ZBJ15}.

Next, we consider the application of Algorithm \ref{PRP-Newton} to Examples \ref{ex:1}--\ref{ex:2} for different values of $m$ and $p$.
In our numerical tests,  `{\tt NCG.}' denotes the total number of CG iterations of the Newton step
at the final iterate of Algorithm \ref{PRP-Newton}. In our numerical tests, we set $\varsigma=10^{-8}$, the parameter pairs $(\zeta_1,\zeta_2)$ are set to be
$(10^{-1},10^{-7})$ and $(10^{-3},10^{-7})$, respectively, and the other parameters and the starting points are set as in section \ref{sec4}. For simplicity, two different pairs of $(\zeta_1,\zeta_2)$ are tested.

Table \ref{table6} displays the numerical results for Example \ref{ex:1} with different values of $m$ and $p$. In Figure \ref{f4}, we  give the convergence history of Algorithm \ref{PRP-Newton} for two tests  of Example \ref{ex:1}  with $(m,p)=(2000,30)$ and $(m,p)=(3000,60)$. Figure \ref{f4} depicts the logarithm of the residual versus the number of iterations for finding a zero of the tangent vector field $F$ defined in (\ref{OJASTIEFEL}).
Table \ref{table8} shows the numerical results for Example \ref{ex:2} with different values of $m$ and $p$.  In Figure \ref{f5}, we give  the convergence history of Algorithm \ref{PRP-Newton} for
for two tests of Example \ref{ex:2} with $(m,p)=(1000,30)$ and $(m,p)=(2000,60)$. Figure \ref{f5} depicts the logarithm of the residual versus the number of iterations for finding a zero of the tangent vector field $F$ defined in (\ref{TRACERATIO}).

We observe from Tables \ref{table6}--\ref{table8} and Figures \ref{f4}--\ref{f5} that, by choosing suitable $\zeta_1$, Algorithm \ref{RDF-PRP} may provide a good initial point for the Riemannian Newton method, which give a high accuracy solution. This shows that the proposed hybrid method is very effective  for solving large-scale problems.
\begin{table}[htbp]
  \caption{Numerical results for Example \ref{ex:1}.}\label{table6}
  \begin{center} {\scriptsize
   \begin{tabular}[c]{|l|l|l|r|c|c|c|c|c|}
     \hline
     \multicolumn{9}{|c|}{$p=30$}\\ \hline
  $m$ & $(\zeta_1,\zeta_2)$ & PRP-Newton  & {\tt CT.}  & {\tt IT.} & {\tt NF.} &  {\tt NCG.} & {\tt Res0.} & {\tt Res.} \\  \hline
\multirow{4}{*}{1000} & \multirow{2}{*}{$(10^{-1},10^{-7})$}
              & PRP Step    & 0.0970 s  &   13  &  20  &      & $1.5637$               & $9.2121\times 10^{-2}$   \\
          &   & Newton Step & 3.5740 s  &   12  &  13  & 1073 & $9.2121\times 10^{-2}$ & $1.2530\times 10^{-8}$   \\ \cline{2-9}
              & \multirow{2}{*}{$(10^{-3},10^{-7})$}
              & PRP Step    & 0.4780 s  &   85  &  92  &      & $1.5637 $ & $9.6138\times 10^{-4}$   \\
         &    & Newton Step & 1.3960 s  &    2  &  3   &  425 & $9.6138\times 10^{-4}$  & $4.2673\times 10^{-10}$   \\  \hline
\multirow{4}{*}{2000} & \multirow{2}{*}{$(10^{-1},10^{-7})$}
              & PRP Step    & 0.1720 s  &   13  &  18  &      & $1.5877 $  & $9.8514\times 10^{-2}$   \\
          &   & Newton Step & 15.9840 s &   16  &  17  & 1761 & $9.8514\times 10^{-2}$  & $1.4495\times 10^{-11}$   \\ \cline{2-9}
              & \multirow{2}{*}{$(10^{-3},10^{-7})$}
              & PRP Step    &  1.1880 s &   95  & 100  &       & $1.5877$  & $9.7528\times 10^{-4}$   \\
         &    & Newton Step &  6.5630 s &    2  &   3  &  705  & $9.7528\times 10^{-4}$  & $9.9379\times 10^{-11}$   \\  \hline
\multirow{4}{*}{3000} & \multirow{2}{*}{$(10^{-1},10^{-7})$}
              & PRP Step    & 0.5150 s  &   13  &  22  &       & $1.5635$  & $8.6702\times 10^{-2}$   \\
          &   & Newton Step & 46.9220 s &   22  &  23  &  1967 & $8.6702\times 10^{-2}$  & $3.6051\times 10^{-8}$   \\ \cline{2-9}
              & \multirow{2}{*}{$(10^{-3},10^{-7})$}
              & PRP Step    &  3.3630 s &  114  & 123  &       & $1.5635$  & $9.9517\times 10^{-4}$   \\
         &    & Newton Step & 14.8350 s &    2  &  3   &   621 & $9.9517\times 10^{-4}$  & $5.1732\times 10^{-13}$   \\  \hline
\multirow{4}{*}{4000} & \multirow{2}{*}{$(10^{-1},10^{-7})$}
              & PRP Step    & 0.7770 s  &   13  &  20  &       & $1.5762$  & $8.6654\times 10^{-2}$   \\
          &   & Newton Step &155.6490 s &   38  &  39  &  3940 & $8.6654\times 10^{-2}$  & $2.3522\times 10^{-8}$   \\ \cline{2-9}
              & \multirow{2}{*}{$(10^{-3},10^{-7})$}
              & PRP Step    &  5.3730 s &  114  & 121  &       & $1.5762$  & $9.9009\times 10^{-4}$   \\
         &    & Newton Step & 33.8150 s &    2  &   3  &   859 & $9.9009\times 10^{-4}$  & $1.6149\times 10^{-8}$   \\  \hline
\multirow{4}{*}{5000} & \multirow{2}{*}{$(10^{-1},10^{-7})$}
              & PRP Step    & 1.1880 s  &    13 &  18  &       & $ 1.5813 $  & $9.4306\times 10^{-2}$   \\
          &   & Newton Step & 315.0050 s&    49 &  50  &  5008 & $9.4306\times 10^{-2}$  & $5.9576\times 10^{-8}$   \\ \cline{2-9}
              & \multirow{2}{*}{$(10^{-3},10^{-7})$}
              & PRP Step    &  13.7560 s&   184 & 189  &       & $1.5813$  & $9.8964\times 10^{-4}$   \\
         &    & Newton Step &136.9180 s &     5 &   6  &  2154 & $9.8964\times 10^{-4}$  & $5.2168\times 10^{-11}$   \\  \hline
     \hline
     \multicolumn{9}{|c|}{$m=3000$}\\ \hline
  $p$ & $(\zeta_1,\zeta_2)$ & PRP-Newton  & {\tt CT.}  & {\tt IT.} & {\tt NF.} &  {\tt NCG.} & {\tt Res0.} & {\tt Res.} \\  \hline
\multirow{4}{*}{20} & \multirow{2}{*}{$(10^{-1},10^{-7})$}
              & PRP Step    & 0.3950 s  & 12 &  21 & & $1.2778$  & $8.5232\times 10^{-2}$   \\
          &   & Newton Step & 48.8860 s & 23 &  24 & 2605 & $8.5232\times 10^{-2}$  & $2.9885\times 10^{-11}$   \\ \cline{2-9}
              & \multirow{2}{*}{$(10^{-3},10^{-7})$}
              & PRP Step    &  2.9630 s & 132 &  141&  & $1.2778$  & $9.8691\times 10^{-4}$   \\
         &    & Newton Step &  15.7150 s & 4 &  5 & 838 & $9.8691\times 10^{-4}$  & $1.2869\times 10^{-8}$   \\  \hline
\multirow{4}{*}{40} & \multirow{2}{*}{$(10^{-1},10^{-7})$}
              & PRP Step    & 0.5300 s  & 13 &  18 & & $1.8237$  & $9.9742\times 10^{-2}$   \\
          &   & Newton Step & 104.7220 s & 30 & 31& 3771 & $9.9742\times 10^{-2}$  & $3.3097\times 10^{-9}$   \\ \cline{2-9}
              & \multirow{2}{*}{$(10^{-3},10^{-7})$}
              & PRP Step    &  3.4950 s & 98 &  103 & & $1.8237$  & $9.6445\times 10^{-4}$   \\
         &    & Newton Step &  38.1140 s & 3 &  4 & 1348 & $9.6445\times 10^{-4}$  & $2.0956\times 10^{-10}$   \\  \hline
\multirow{4}{*}{60} & \multirow{2}{*}{$(10^{-1},10^{-7})$}
              & PRP Step    & 0.9530 s  & 15 &  30 & & $2.2435$  & $9.6070\times 10^{-2}$   \\
          &   & Newton Step & 63.3580 s & 15 &  16 & 1778 & $9.6070\times 10^{-2}$  & $6.9081\times 10^{-10}$   \\ \cline{2-9}
              & \multirow{2}{*}{$(10^{-3},10^{-7})$}
              & PRP Step    &  6.4520 s & 121 &  136&  & $2.2435$  & $9.8044\times 10^{-4}$   \\
         &    & Newton Step &  27.8180 s & 2 &  3 & 726 & $9.8044\times 10^{-4}$  & $4.1246\times 10^{-9}$   \\  \hline
\multirow{4}{*}{80} & \multirow{2}{*}{$(10^{-1},10^{-7})$}
              & PRP Step    & 1.5960 s  & 16 &  27&  & $2.5407$  & $9.9187\times 10^{-2}$   \\
          &   & Newton Step &  128.4430 s & 22 &  23 & 2445 & $9.9187\times 10^{-2}$  & $1.6713\times 10^{-9}$   \\ \cline{2-9}
              & \multirow{2}{*}{$(10^{-3},10^{-7})$}
              & PRP Step    & 12.2660 s & 159 &  170&  & $2.5407$  & $9.8479\times 10^{-4}$   \\
         &    & Newton Step & 56.0430 s & 3 &  4& 1062  & $9.8479\times 10^{-4}$  & $1.0519\times 10^{-11}$   \\  \hline
\multirow{4}{*}{100} & \multirow{2}{*}{$(10^{-1},10^{-7})$}
              & PRP Step    & 2.1910 s  & 17 &  30  & & $2.9047$  & $9.6005\times 10^{-2}$   \\
          &   & Newton Step & 251.5150 s & 25 &  26& 3773 & $9.6005\times 10^{-2}$  & $1.9501\times 10^{-11}$   \\ \cline{2-9}
              & \multirow{2}{*}{$(10^{-3},10^{-7})$}
              & PRP Step    &  11.0250 s & 105 &  118 & & $2.9047$  & $9.9363\times 10^{-4}$   \\
         &    & Newton Step &  76.9210 s & 3 &  4& 1148 & $9.9363\times 10^{-4}$  & $4.1883\times 10^{-8}$   \\  \hline
  \end{tabular} }
  \end{center}
\end{table}
\begin{figure}[htbp]
\begin{center}
\begin{tabular}{cc}
\epsfig{figure=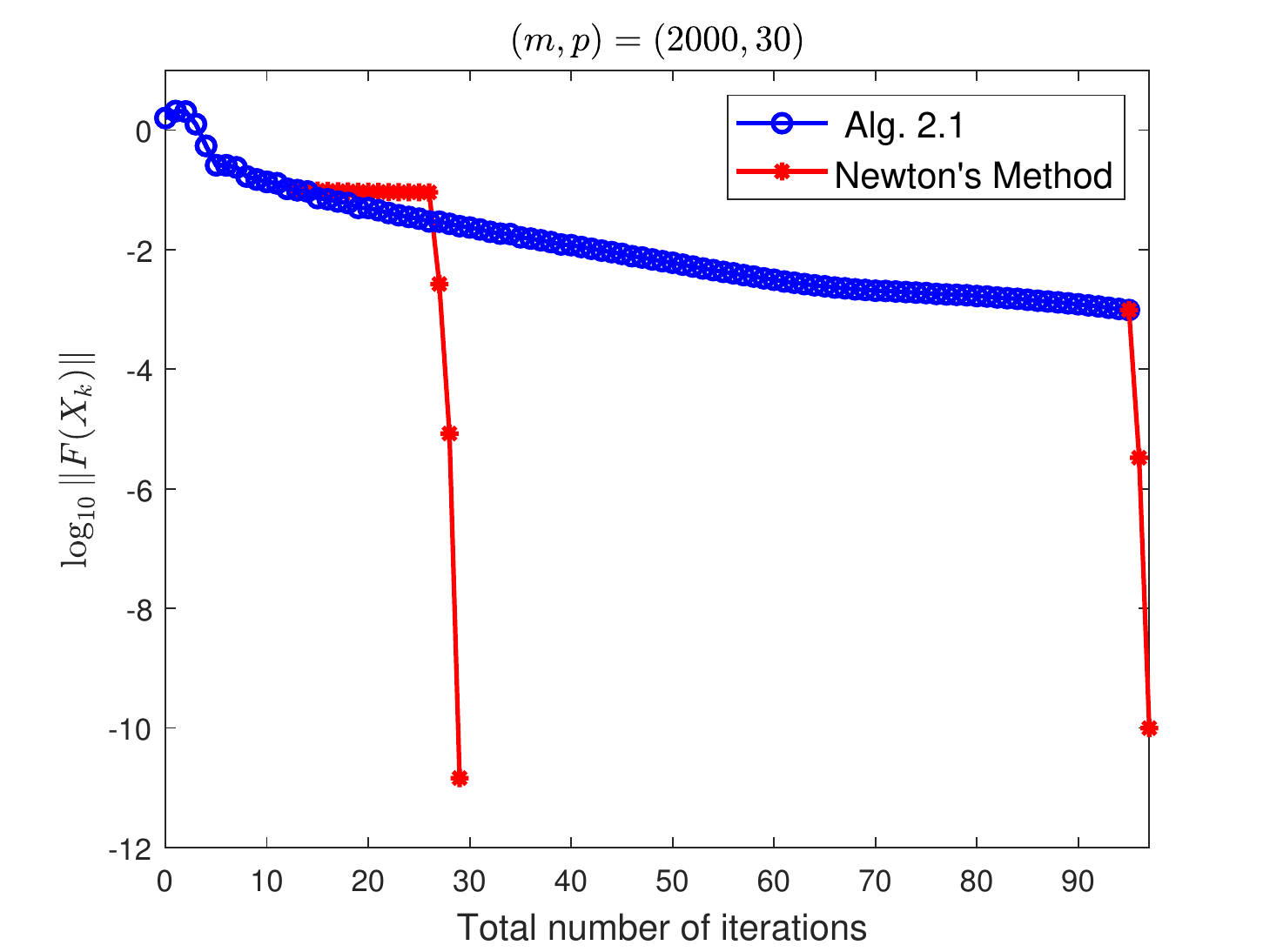,width=6.2cm} & \epsfig{figure=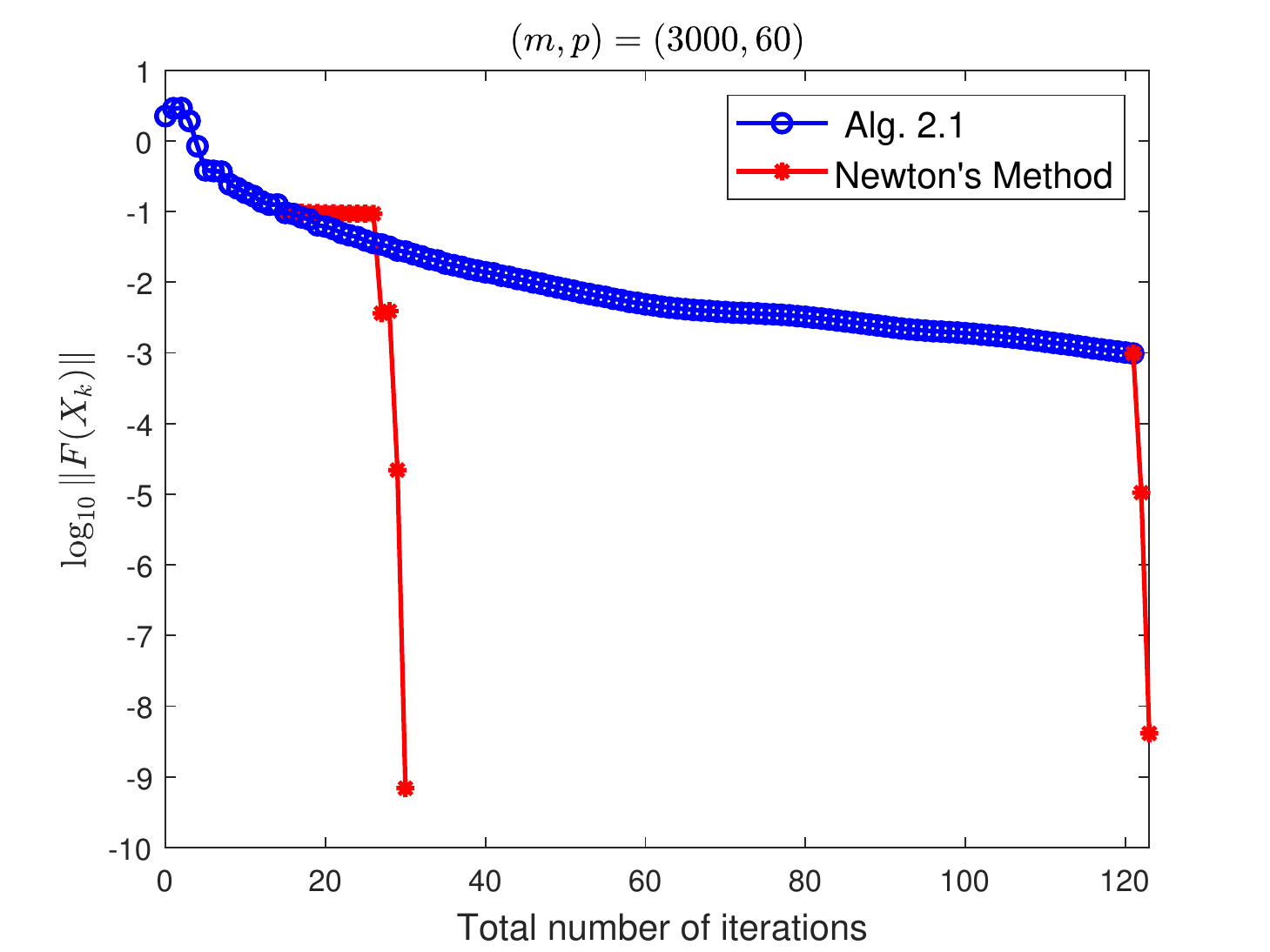,width=6.2cm}
\end{tabular}
\end{center}
 \caption{Convergence history of two tests for Example \ref{ex:1}.}
 \label{f4}
\end{figure}

\begin{table}[htbp]
  \caption{Numerical results for Example \ref{ex:2}.}\label{table8}
  \begin{center} {\scriptsize
   \begin{tabular}[c]{|l|l|l|r|c|c|c|c|c|}
     \hline
     \multicolumn{9}{|c|}{$p=30$}\\ \hline
  $m$ & $(\zeta_1,\zeta_2)$ & PRP-Newton  & {\tt CT.}  & {\tt IT.} & {\tt NF.} &  {\tt NCG.} & {\tt Res0.} & {\tt Res.} \\  \hline
\multirow{4}{*}{1000} & \multirow{2}{*}{$(10^{-1},10^{-7})$}
              & PRP Step    & 1.2650 s  & 81 &  98&  & $1.2016\times 10^{2}$  & $9.8709\times 10^{-2}$   \\
          &   & Newton Step & 4.6090 s & 3 &  4& 629 & $9.8709\times 10^{-2}$  & $4.9052\times 10^{-10}$   \\ \cline{2-9}
              & \multirow{2}{*}{$(10^{-3},10^{-7})$}
              & PRP Step    &  2.5630 s & 161 &  178&  & $1.2016\times 10^{2}$  & $9.5649\times 10^{-4}$   \\
         &    & Newton Step &  1.5150 s & 1 &  2 & 189 & $9.5649\times 10^{-4}$  & $2.0833\times 10^{-10}$   \\  \hline
\multirow{4}{*}{2000} & \multirow{2}{*}{$(10^{-1},10^{-7})$}
              & PRP Step    & 5.9340 s  & 90 &  105 & & $1.7128\times 10^{2}$  & $9.8977\times 10^{-2}$   \\
          &   & Newton Step & 29.3860 s & 3 &  4 & 855 & $9.8977\times 10^{-2}$  & $2.4162\times 10^{-9}$   \\ \cline{2-9}
              & \multirow{2}{*}{$(10^{-3},10^{-7})$}
              & PRP Step    &  10.9610 s & 174 &  189 &  & $1.7128\times 10^{2}$  & $9.8098\times 10^{-4}$   \\
         &    & Newton Step &  9.1750 s & 1 &  2 & 265 & $9.8098\times 10^{-4}$  & $6.9162\times 10^{-10}$   \\  \hline
\multirow{4}{*}{3000} & \multirow{2}{*}{$(10^{-1},10^{-7})$}
              & PRP Step    & 13.8460 s  & 98 &  111 & & $2.0983\times 10^{2}$  & $9.9004\times 10^{-2}$   \\
          &   & Newton Step & 64.1920 s & 3 &  4 & 860 & $9.9004\times 10^{-2}$  & $7.2418\times 10^{-11}$   \\ \cline{2-9}
              & \multirow{2}{*}{$(10^{-3},10^{-7})$}
              & PRP Step    &  26.5980 s & 197 &  210 & & $2.0983\times 10^{2}$  & $9.4199\times 10^{-4}$   \\
         &    & Newton Step &  20.1030 s & 1 &  2 & 260 & $9.4199\times 10^{-4}$  & $8.3461\times 10^{-11}$   \\  \hline
\multirow{4}{*}{4000} & \multirow{2}{*}{$(10^{-1},10^{-7})$}
              & PRP Step    & 34.0490 s  & 138 &  155 & & $2.4355\times 10^{2}$  & $9.9053\times 10^{-2}$   \\
          &   & Newton Step & 110.148 s & 3 &  4 & 850 & $9.9053\times 10^{-2}$  & $1.4005\times 10^{-9}$   \\ \cline{2-9}
              & \multirow{2}{*}{$(10^{-3},10^{-7})$}
              & PRP Step    & 67.5570 s & 289 &  306 & & $2.4355\times 10^{2}$  & $9.7171\times 10^{-4}$   \\
         &    & Newton Step & 34.1720 s & 1 &  2 & 257 & $9.7171\times 10^{-4}$  & $1.3732\times 10^{-10}$   \\  \hline
\multirow{4}{*}{5000} & \multirow{2}{*}{$(10^{-1},10^{-7})$}
              & PRP Step    & 37.3010 s  & 91 &  116 & & $2.7389\times 10^{2}$  & $9.8258\times 10^{-2}$   \\
          &   & Newton Step &  214.2700 s & 3 &  4 & 936 & $9.8258\times 10^{-2}$  & $1.1501\times 10^{-9}$   \\ \cline{2-9}
              & \multirow{2}{*}{$(10^{-3},10^{-7})$}
              & PRP Step    &  75.7800 s & 199 &  224 &  & $2.7389\times 10^{2}$  & $9.6717\times 10^{-4}$   \\
         &    & Newton Step &  54.7750 s & 1 &  2 & 276 & $9.6717\times 10^{-4}$  & $1.9560\times 10^{-11}$   \\  \hline
     \hline
     \multicolumn{9}{|c|}{$m=2000$}\\ \hline
  $p$ & $(\zeta_1,\zeta_2)$ & PRP-Newton  & {\tt CT.}  & {\tt IT.} & {\tt NF.} &  {\tt NCG.} & {\tt Res0.} & {\tt Res.} \\  \hline
\multirow{4}{*}{20} & \multirow{2}{*}{$(10^{-1},10^{-7})$}
              & PRP Step    & 5.8470 s  & 95 &  121 & & $1.4046\times 10^{2}$  & $9.5643\times 10^{-2}$   \\
          &   & Newton Step & 14.6980 s & 2 &  3 & 461 & $9.5643\times 10^{-2}$  & $4.7077\times 10^{-9}$   \\ \cline{2-9}
              & \multirow{2}{*}{$(10^{-3},10^{-7})$}
              & PRP Step    & 14.6830 s & 199 &  225 &  & $1.4046\times 10^{2}$  & $9.3379\times 10^{-4}$   \\
         &    & Newton Step & 9.2240 s & 1 &  2 & 239 & $9.3379\times 10^{-4}$  & $2.2436\times 10^{-10}$   \\  \hline
\multirow{4}{*}{40} & \multirow{2}{*}{$(10^{-1},10^{-7})$}
              & PRP Step    & 7.9220 s  & 109 &  128&  & $1.9775\times 10^{2}$  & $9.9096\times 10^{-2}$   \\
          &   & Newton Step & 65.6650 s & 3 &  4 & 1799 & $9.9096\times 10^{-2}$  & $5.1574\times 10^{-10}$   \\ \cline{2-9}
              & \multirow{2}{*}{$(10^{-3},10^{-7})$}
              & PRP Step    & 16.3930 s & 235 &  254 & & $1.9775\times 10^{2}$  & $9.8092\times 10^{-4}$   \\
         &    & Newton Step & 8.9040 s & 1 &  2 & 240 & $9.8092\times 10^{-4}$  & $1.2536\times 10^{-10}$   \\  \hline
\multirow{4}{*}{60} & \multirow{2}{*}{$(10^{-1},10^{-7})$}
              & PRP Step    &  7.2970 s  &  95 &  114 & & $2.4225\times 10^{2}$  & $9.7040\times 10^{-2}$   \\
          &   & Newton Step & 46.2810 s & 4 &  5 & 1142 & $9.7040\times 10^{-2}$  & $2.5144\times 10^{-10}$   \\ \cline{2-9}
              & \multirow{2}{*}{$(10^{-3},10^{-7})$}
              & PRP Step    &  13.6400 s & 179 &  198 & & $2.4225\times 10^{2}$  & $9.8101\times 10^{-4}$   \\
         &    & Newton Step &   9.8130 s & 1 &  2 & 240 & $9.8101\times 10^{-4}$  & $1.5623\times 10^{-10}$   \\  \hline
\multirow{4}{*}{80} & \multirow{2}{*}{$(10^{-1},10^{-7})$}
              & PRP Step    & 11.2390 s  & 94 &  111& & $2.7770\times 10^{2}$  & $9.9043\times 10^{-2}$   \\
          &   & Newton Step & 50.2640 s & 3 &  4 & 847 & $9.9043\times 10^{-2}$  & $6.0549\times 10^{-10}$   \\ \cline{2-9}
              & \multirow{2}{*}{$(10^{-3},10^{-7})$}
              & PRP Step    & 22.1340 s & 193 &  210 & & $2.7770\times 10^{2}$  & $9.6451\times 10^{-4}$   \\
         &    & Newton Step & 15.2620 s & 1 &  2 & 252  & $9.6451\times 10^{-4}$  & $1.0738\times 10^{-10}$   \\  \hline
\multirow{4}{*}{100} & \multirow{2}{*}{$(10^{-1},10^{-7})$}
              & PRP Step    & 15.1400 s  & 104 &  121 & & $3.0711\times 10^{2}$  & $9.7853\times 10^{-2}$   \\
          &   & Newton Step & 148.2740 s & 6 &  7 & 2104& $9.7853\times 10^{-2}$  & $2.1292\times 10^{-11}$   \\ \cline{2-9}
              & \multirow{2}{*}{$(10^{-3},10^{-7})$}
              & PRP Step    & 28.7670 s & 212 &  229 & & $3.0711\times 10^{2}$  & $9.7389\times 10^{-4}$   \\
         &    & Newton Step & 18.3350 s & 1 &  2 & 257 & $9.7389\times 10^{-4}$  & $4.2447\times 10^{-10}$   \\  \hline
  \end{tabular} }
  \end{center}
\end{table}
\begin{figure}[htbp]
\begin{center}
\begin{tabular}{cc}
\epsfig{figure=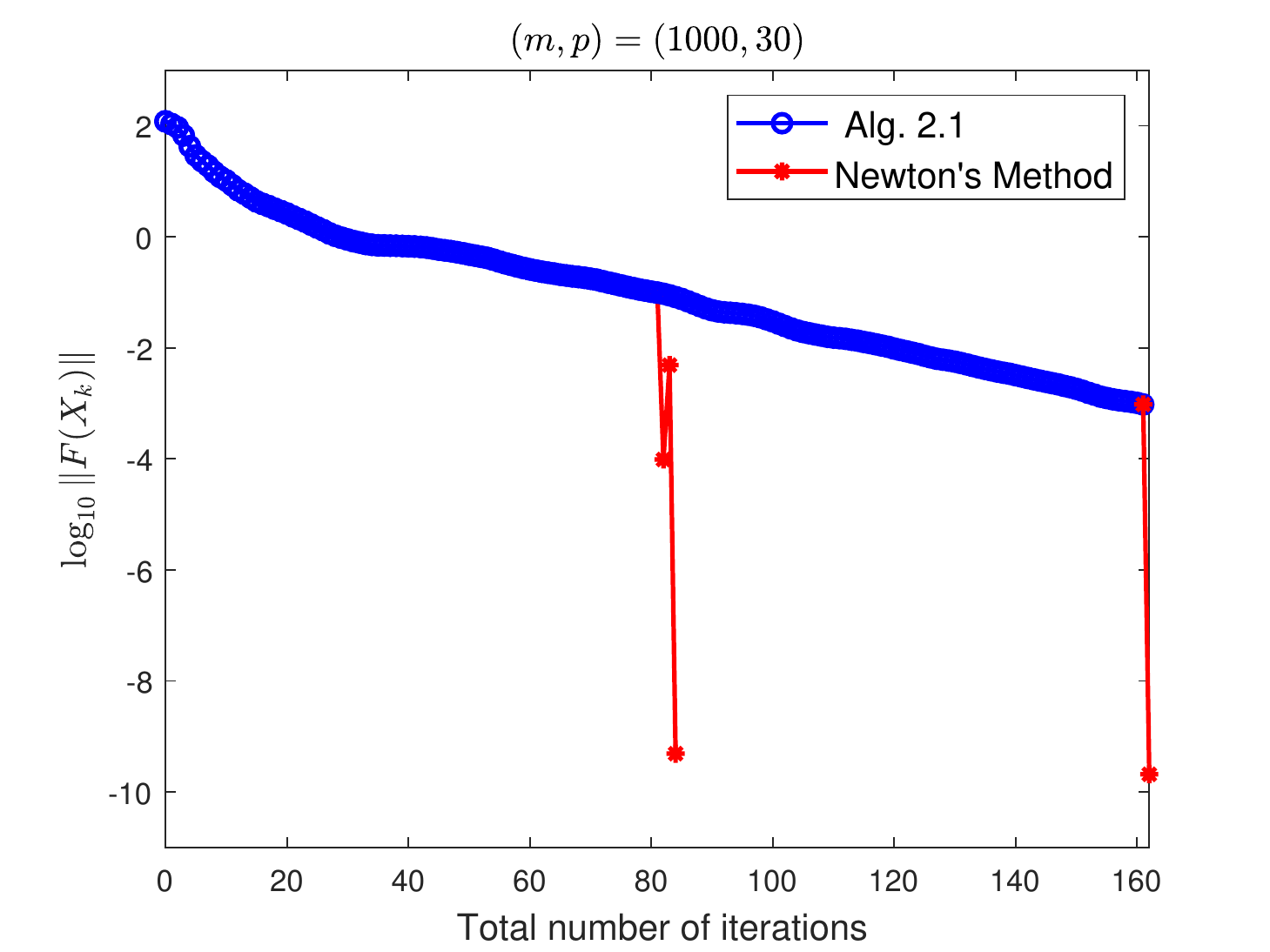,width=6.2cm} & \epsfig{figure=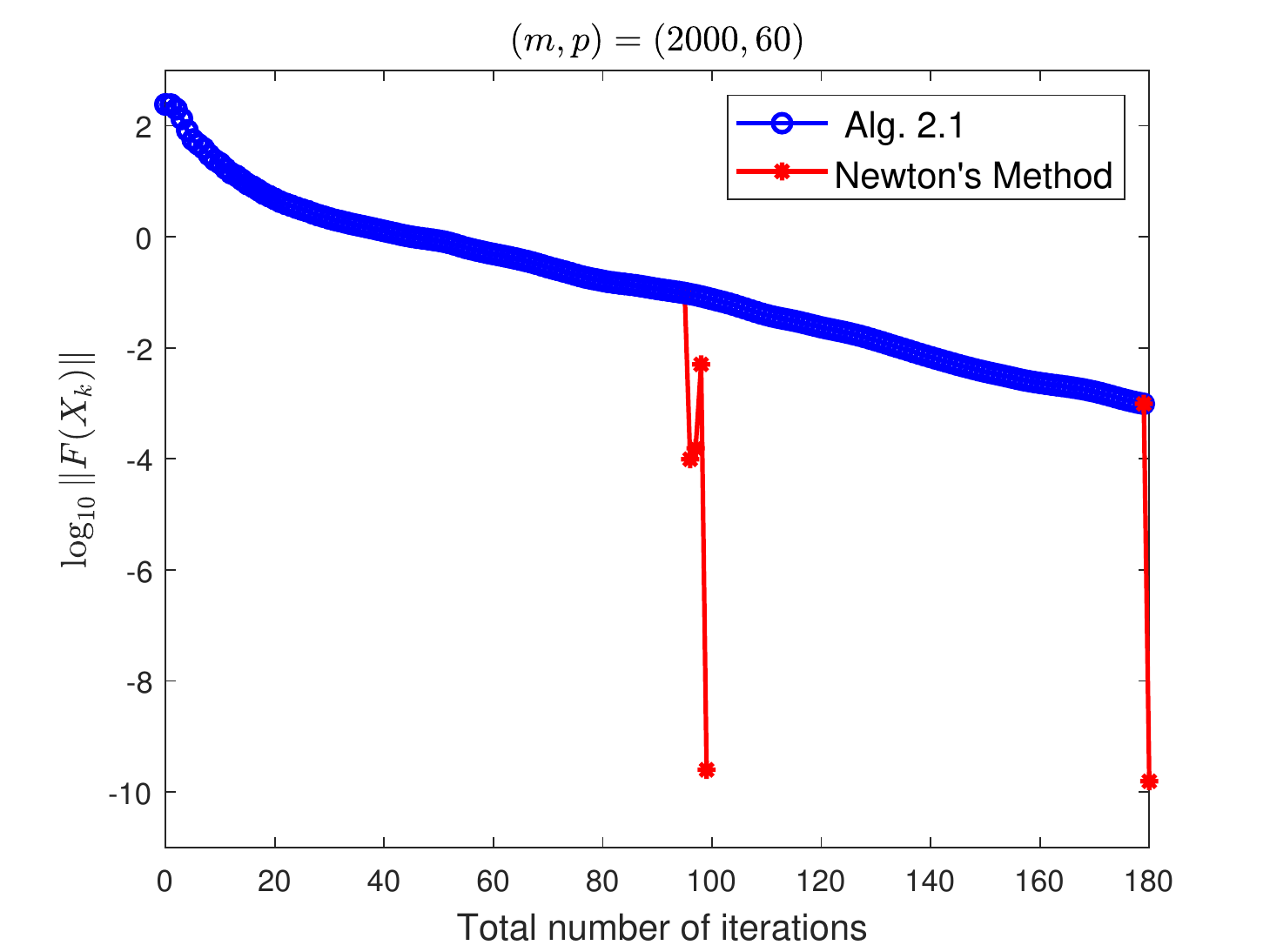,width=6.2cm}
\end{tabular}
\end{center}
 \caption{Convergence history of two tests for Example \ref{ex:2}.}
 \label{f5}
\end{figure}

\section{Conclusions}\label{sec6}
In this paper, we have proposed a Riemannian Derivative-Free PRP Method  for finding a zero of a tangent vector field on a Riemannian manifold.  By using a non-monotone line search, the global convergence of the proposed geometric method is established under some mild conditions.
To further improve the efficiency, we also provide a hybrid method, which combines the proposed
geometric algorithm with the Riemannian Newton method. Numerical tests illustrate the efficiency of the proposed geometric algorithm for large-scale problems. An interesting question is how to choose the stopping tolerance $\zeta_1$ such that the overall computational cost of  Algorithm \ref{PRP-Newton} is minimized, which needs further study.


\end{document}